\documentclass{article}
\usepackage{graphicx}
\usepackage{amssymb}
\usepackage{amsmath}
\usepackage{amsthm}
\usepackage{amscd}
\usepackage[dvipsnames]{xcolor}
\newtheorem{theorem}{Theorem}[section]
\newtheorem{lemma}[theorem]{Lemma}
\newtheorem{corollary}[theorem]{Corollary}
\newtheorem{proposition}[theorem]{Proposition}
\theoremstyle{definition}
\newtheorem{remark}[theorem]{Remark}

\newtheorem{definition}[theorem]{Definition}
\numberwithin{equation}{section}

\newcommand{\red}[1]{{\color{red}#1}}
\newcommand{\blue}[1]{{\color{blue}#1}}

\newcommand{\ignore}[1]{}
\newcommand{\bli}{\begin{list}{}{\labelwidth6mm\leftmargin8mm}}
\newcommand{\eli}{\end{list}}

\newcommand{\dint}{\;\mathrm{d}}
\newcommand{\Dd}{\;\mathrm{D}}
\newcommand{\whole}[1]{\ensuremath\left\lfloor #1 \right\rfloor}

\newcommand{\fz}{\infty}
\def\ls{\lesssim}
\def\vz{\varphi}
\def\supp{{\mathop\mathrm{supp\,}\nolimits}}
\newcommand{\Be}{{B}_{p_1,q_1}^{s_1}}
\newcommand{\Bz}{{B}_{p_2,q_2}^{s_2}}
\newcommand{\Fe}{{F}_{p_1,q_1}^{s_1}}
\newcommand{\Fz}{{F}_{p_2,q_2}^{s_2}}
\newcommand{\Ae}{{A}_{p_1,q_1}^{s_1}}
\newcommand{\Az}{{A}_{p_2,q_2}^{s_2}}
\newcommand{\bt}{{B}_{p,q}^{s,\tau}}
\newcommand{\bte}{{B}_{p_1,q_1}^{s_1,\tau_1}}
\newcommand{\btz}{{B}_{p_2,q_2}^{s_2,\tau_2}}
\newcommand{\at}{{A}_{p,q}^{s,\tau}}
\newcommand{\ate}{{A}_{p_1,q_1}^{s_1,\tau_1}}
\newcommand{\atz}{{A}_{p_2,q_2}^{s_2,\tau_2}}
\newcommand{\ft}{{F}_{p,q}^{s,\tau}}
\newcommand{\fte}{{F}_{p_1,q_1}^{s_1,\tau_1}}
\newcommand{\ftz}{{F}_{p_2,q_2}^{s_2,\tau_2}}
\newcommand{\MA}{\ensuremath{{\cal A}^{s}_{u,p,q}}}

\newcommand{\MB}{\ensuremath{{\cal N}^{s}_{u,p,q}}}
\newcommand{\MBe}{\ensuremath{{\cal N}^{s_1}_{u_1,p_1,q_1}}}
\newcommand{\MBz}{\ensuremath{{\cal N}^{s_2}_{u_2,p_2,q_2}}}
\newcommand{\MF}{\ensuremath{{\cal E}^{s}_{u,p,q}}}

\newcommand{\M}{\ensuremath{{\cal M}_{u,p}}}
\newcommand{\cM}{\ensuremath{{\mathcal M}}}
\newcommand{\SRn}{\mathcal{S}(\rn)}
\newcommand{\SpRn}{\mathcal{S}'(\rn)}
\newcommand{\sbt}{{b}_{p,q}^{s,\tau}}
\renewcommand{\sbt}{{b}_{p,q}^{s,\tau}}
\newcommand{\sbte}{{b}_{p_1,q_1}^{s_1,\tau_1}}
\newcommand{\sbtz}{{b}_{p_2,q_2}^{s_2,\tau_2}}

\newcommand{\sat}{{a}_{p,q}^{s,\tau}}
\newcommand{\sft}{{f}_{p,q}^{s,\tau}}
\newcommand{\bmo}{\mathrm{bmo}}
\newcommand{\Lloc}{L_1^{\mathrm{loc}}}
\newcommand{\B}{\ensuremath{B^s_{p,q}}}
\newcommand{\F}{\ensuremath{F^s_{p,q}}}
\newcommand{\A}{\ensuremath{A^s_{p,q}}}
\newcommand{\beq}{\begin{equation}}
\newcommand{\eeq}{\end{equation}}
\newcommand{\ve}{\varepsilon}

\newcommand{\mixB}{\ell_q(L_p^\tau)}
\newcommand{\mixF}{L_p^\tau(\ell_q)}

\newcommand{\id}{\ensuremath{\operatorname{id}}}

\newcommand{\real}{\ensuremath{{\mathbb R}}}
\newcommand{\rr}{{\real}}
\newcommand{\rn}{\ensuremath{{\real^d}}}
\newcommand{\cc}{\ensuremath{{\mathbb C}}}
\newcommand{\zz}{\ensuremath{{\mathbb Z}}}
\newcommand{\zn}{\ensuremath{{\zz^d}}}
\newcommand{\nn}{\ensuremath{{\mathbb N}}}
\newcommand{\no}{\ensuremath{{\mathbb N}_0}}
\newcommand{\cs}{\ensuremath{\mathcal S}}

\newcommand{\critical}{\ensuremath{\gamma(\tau_1,\tau_2,p_1,p_2)}}

\newcommand{\ext}{\ensuremath{\mathop\mathrm{ext}\nolimits}}
\newcommand{\Ext}{\ensuremath{\mathop\mathrm{Ext}\nolimits}}
\newcommand{\re}{\ensuremath{\mathop\mathrm{re}}}

\begin{document}

\title{Limiting embeddings of Besov-type and Triebel-Lizorkin-type spaces on domains and an extension operator}

\author{Helena F. Gon\c{c}alves, Dorothee D. Haroske and Leszek Skrzypczak\footnotemark[1]}







\footnotetext[1]{All authors were partially supported by the German Research Foundation (DFG), Grant no. Ha 2794/8-1. The third author was also supported by National Science Center, Poland,  Grant no. 2013/10/A/ST1/00091.}

\maketitle

	\begin{abstract} 
In this paper, we study limiting embeddings of Besov-type and Triebel-Lizorkin-type spaces, $\id_\tau: \bte(\Omega) \hookrightarrow \btz(\Omega)$ and $\id_\tau : \fte(\Omega)  \hookrightarrow \ftz(\Omega)$, where $\Omega \subset \rn$ is a bounded domain, obtaining necessary and sufficient conditions for the continuity of $\id_\tau$. This can also be seen as the continuation of our previous studies of compactness of the embeddings in the non-limiting case. 
Moreover, 
we also construct Rychkov's  linear, bounded universal extension operator for these spaces.\\

	 
\end{abstract}

{\section{Introduction}\label{intro}
  Besov-type spaces $\bt(\rn)$ and Triebel-Lizorkin-type spaces $\ft(\rn)$, $0 < p <\infty$ (or $p=\infty$ in the $B$-case),
   $0 < q \le \infty$, $\tau\ge 0$, $s \in \real$, are part of a class of function spaces built upon Morrey spaces $\M(\rn)$, $0< p \le u < \infty$. They are regularly called in the literature as \textit{smoothness spaces of Morrey type} or, shortly, \textit{smoothness Morrey spaces}, and they have been increasingly studied in the last decades, motivated firstly by possible applications.

The classical Morrey spaces $\M$, $0< p \le u < \infty$ , were introduced by Morrey in \cite{Mor} and are part of a wider class of Morrey-Campanato spaces, cf. \cite{Pee}. They can be seen as a complement to $L_p$ spaces, since $\ensuremath{{\cal M}_{p,p}}(\rn)= L_p(\rn)$.

The (inhomogeneous) Besov-type and Triebel-Lizorkin-type spaces we work here with were introduced and intensively studied in \cite{ysy} by Yuan, Sickel and Yang. 
Their homogeneous versions were previously investigated by El Baraka in \cite{ElBaraka1,ElBaraka2, ElBaraka3}, and also by Yuan and Yang \cite{yy1,yy2}.  Considering $\tau=0$, one recovers the classical Besov and Triebel-Lizorkin spaces. Moreover, they are also closely connected with Besov-Morrey spaces $\MB(\rn)$ and Triebel-Lizorkin-Morrey spaces $\MF(\rn)$, $0 < p \le u < \infty$, $0 < q \le \infty$, $s \in \real$, which are also included in the class of \textit{smoothness Morrey spaces}. Namely, when $p\leq u$, $\tau=\frac{1}{p}-\frac{1}{u}$, then the Triebel-Lizorkin-type space $\ft(\rn)$ coincides with $\MF(\rn)$, and it is also known that the Besov-Morrey space $\MB(\rn)$ is a proper subspace of $\bt(\rn)$ with $\tau=\frac{1}{p}-\frac{1}{u}$, $p<u$ and $q<\infty$. The Besov-Morrey spaces were introduced by Kozono and Yamazaki in \cite{KY} and used by them and later on by Mazzucato \cite{Maz} in the study of Navier-Stokes equations. In \cite{TX} Tang and Xu introduced the corresponding Triebel-Lizorkin-Morrey spaces, thanks to establishing the Morrey version of the  Fefferman-Stein vector-valued inequality. Some properties of these spaces including their wavelet characterisations were later described in the papers by Sawano \cite{Saw2,Saw1}, Sawano and Tanaka \cite{ST2,ST1} and Rosenthal \cite{MR-1}.
{The surveys \cite{s011,s011a} by Sickel are also worth of being consulted when studying these scales.} 
 Recently, some limiting embedding properties of these spaces were investigated in a series of papers \cite{hs12,hs12b,hs14, HaSk-krakow, HaSk-morrey-comp}. As for the Besov-type and Triebel-Lizorkin-type spaces, also embedding properties have been recently studied in \cite{ghs20, YHMSY, YHSY}.

\bigskip

Undoubtedly the question of necessary and sufficient conditions for continuous embeddings of certain function spaces is a natural and classical one. Beyond that, this paper should essentially be understood as the continuation of our earlier studies in \cite{ghs20}. Proceeding contrary to the usual, there we started by studying compactness of the embeddings of Besov-type ($A=B$) and Triebel-Lizorkin-type ($A=F$) spaces,
\[\id_\tau~:~\ate(\Omega)~\hookrightarrow~\atz(\Omega),\]
where $\Omega \subset \rn$ is a bounded domain. Now we finally deal with the continuity of such embeddings, obtaining sufficient and necessary conditions on the parameters under which they hold true.  According to the results obtained in \cite{ghs20}, the embedding $\id_\tau$ is compact if, and only if, 
\[
\frac{s_1-s_2}{d} > \max \left\{ \left(\tau_2-\frac{1}{p_2} \right)_+- \left(\tau_1-\frac{1}{p_1} \right)_+, \frac{1}{p_1}-\tau_1 -\min\left\{ \frac{1}{p_2}-\tau_2, \frac{1}{p_2}(1-p_1\tau_1)_+\right\} \right\}=: \gamma, 
\]
with $a_+:= \max\{a,0\}$, and there is no continuous embedding $\id_\tau$ when $s_1-s_2<d \, \gamma$. Consequently, only the case when
$$
\frac{s_1-s_2}{d} = \gamma
$$
is of interest to us here. In what follows, we call this setting `limiting situation', giving meaning to the expression `limiting embedding'. In that way we complement earlier results in \cite{hs12b,hs14,YHSY} in related settings. 
{Our main results are stated in Theorem~\ref{lim-lim},  Propositions~\ref{lim-tau2-large} and \ref{lim-tau1-large}. In the propositions we consider the situation when one of the spaces, the {source} one or the target one, coincides  with some classical Besov space $B^\sigma_{\infty,\infty }(\Omega)$. The outcome for the spaces that does not satisfy this assumption can be found in the theorem.  In almost all cases we prove the sharp sufficient and necessary conditions. Only in one case we have a small gap between them, cf. Remark~\ref{lim-rem}. }     } 

Of independent interest  is also the extension theorem we are able to prove for the spaces under consideration. The first extension operators for Besov-type and Triebel-Lizorkin type spaces were constructed by Sickel, Yang and Yuan in \cite{ysy}, cf. Theorem 6.11 and 6.13 ibidem. However it is assumed there that domains are $C^\infty$ smooth and the extension operators were not universal. Another not universal construction for the smooth domains was  given by Moura, Neves and Schneider in \cite{MNS}. The Rychkov universal extension operator for the Triebel-Lizorkin type spaces defined on Lipschitz domains was recently constructed by  Zhou, Hovemann and Sickel in \cite{ZHS20} with additional assumption   $p,q \in [1, \infty)$. Here we considered  all admissible parameters $p$ and $q$. We concentrate on the Besov-type spaces, that  are not a real interpolation space of Triebel-Lizorkin-type spaces, in contrast to the classical case, cf.  \cite{YSY-15}. 
 The extension theorem will not only help us to obtain the results about the continuity, but also will allow us to improve some necessary conditions of such embeddings on $\rn$. We follow \ignore{the approach of Rychkov in} {Rychkov's approach from} \cite{Ryc} and construct such an operator, for all possible values of \ignore{$p\in(0, \infty]$} {$p,q\in(0, \infty]$}. We learned only recently that in \cite{Z21,ZHS20} the authors followed a similar approach to construct such an extension operator adapted to their purposes, that is, for $p,q\geq 1$. For the convenience of the reader we keep our argument for the full range of parameters here.

This paper is organized as follows. In Section~\ref{prelim} we recall the definition, on $\rn$ and on bounded domains $\Omega \subset \rn$, of the spaces considered in the paper and collect some basic properties, among them the   wavelet characterisations. In Section~\ref{extension} we deal with the construction of a universal linear bounded extension operator for the spaces $\at$. Section~\ref{lim-emb} is, finally, devoted to the study of continuity properties of limiting embeddings of the spaces $\at(\Omega)$. Moreover, we make use of these results and the extension theorem from Section 3 to improve prior results on the continuity of embeddings of the corresponding spaces on $\rn$, cf. \cite{YHSY}.

\section{Preliminaries}\label{prelim}
First we fix some notation. By $\nn$ we denote the \emph{set of natural numbers},
by $\nn_0$ the set $\nn \cup \{0\}$,  and by $\zn$ the \emph{set of all lattice points
	in $\rn$ having integer components}. {Let $\no^d$, where $d \in \nn$, be the set of all multi-indices, $\alpha=(\alpha_1,..., \alpha_d)$ with $\alpha_j \in \no$ and $|\alpha|:=\sum_{j=1}^d\alpha_j$. If $x=(x_1,..., x_d) \in \rn$ and $\alpha=(\alpha_1,..., \alpha_d) \in \no^d$, then we put $x^\alpha=x_1^{\alpha_1} \cdot \cdot \cdot x_d^{\alpha_d}$.}
For $a\in\real$, let   $\whole{a}:=\max\{k\in\zz: k\leq a\}$ and $a_+:=\max\{a,0\}$.
All unimportant positive constants will be denoted by $C$, occasionally with
subscripts. By the notation $A \ls B$, we mean that there exists a positive constant $C$ such that
$A \le C \,B$, whereas  the symbol $A \sim B$ stands for $A \ls B \ls A$.
We denote by $B(x,r) :=  \{y\in \rn: |x-y|<r\}$ the ball centred at $x\in\rn$ with radius $r>0$, and $|\cdot|$ denotes the Lebesgue measure when applied to measurable subsets of $\rn$.

Given two (quasi-)Banach spaces $X$ and $Y$, we write $X\hookrightarrow Y$
if $X\subset Y$ and the natural embedding of $X$ into $Y$ is continuous.

\subsection{Smoothness spaces of Morrey type on $\rn$}

Let $\SRn$ be the set of all \emph{Schwartz functions} on $\rn$, endowed
with the usual topology,
and denote by $\SpRn$ its \emph{topological dual}, namely,
the space of all bounded linear functionals on $\SRn$
endowed with the weak $\ast$-topology.
For all $f\in \cs(\rn)$ or $\cs'(\rn)$, we
use $\widehat{f}$ to denote its \emph{Fourier transform}, and $f^\vee$ for its inverse.
Let $\mathcal{Q}$ be the collection of all \emph{dyadic cubes} in $\rn$, namely,
$
\mathcal{Q}:= \{Q_{j,k}:= 2^{-j}([0,1)^d+k):\ j\in\zz,\ k\in\zn\}.
$
The {symbol}  $\ell(Q)$ denotes
the side-length of the cube $Q$ and $j_Q:=-\log_2\ell(Q)$. {Moreover, we denote by $\chi_{Q_{j,m}}$ the characteristic function of the cube $Q_{j,m}$.} \\

Let $\vz_0,$ $\vz\in\SRn$. We say that $(\vz_0, \vz)$ is an \textit{admissible pair} if
\begin{equation}\label{e1.0}
\supp \widehat{\vz_0}\subset \{\xi\in\rn:\,|\xi|\le2\}\, , \qquad
|\widehat{\vz_0}(\xi)|\ge C\ \text{if}\ |\xi|\le 5/3,
\end{equation}
and
\begin{equation}\label{e1.1}
\supp \widehat{\vz}\subset \{\xi\in\rn: 1/2\le|\xi|\le2\}\quad\text{and}\quad
|\widehat{\vz}(\xi)|\ge C\ \text{if}\  3/5\le|\xi|\le 5/3,
\end{equation}
where $C$ is a positive constant.
In what follows, for all $\vz\in\cs(\rn)$ and $j\in\nn$, $\vz_j(\cdot):=2^{jd}\vz(2^j\cdot)$.\\

\begin{definition}\label{d1}
Let $s\in\rr$, $\tau\in[0,\infty)$, $q \in(0,\fz]$ and $(\vz_0,\vz)$ be an admissible pair. 
\bli
\item[{\bfseries\upshape (i)}]
Let $p\in(0,\infty]$. The \emph{Besov-type space} $\bt(\rn)$ is defined to be the collection of all $f\in\SpRn$ such that

$$\|f \mid {\bt(\rn)}\|:=
\sup_{P\in\mathcal{Q}}\frac1{|P|^{\tau}}\left\{\sum_{j=\max\{j_P,0\}}^\fz\!\!
2^{js q}\left[\int\limits_P
|\vz_j\ast f(x)|^p\dint x\right]^{\frac{q}{p}}\right\}^{\frac1q}<\fz$$

with the usual modifications made in case of $p=\fz$ and/or $q=\fz$.
\item[{\bfseries\upshape (ii)}]
Let $p\in(0,\infty)$. The \emph{Triebel-Lizorkin-type space} $\ft(\rn)$ is defined to be the collection of all $f\in \SpRn$ such that
$$\|f \mid {\ft(\rn)}\|:=
\sup_{P\in\mathcal{Q}}\frac1{|P|^{\tau}}\left\{\int\limits_P\left[\sum_{j=\max\{j_P,0\}}^\fz\!\!
2^{js q}
|\vz_j\ast f(x)|^q\right]^{\frac{p}{q}}\dint x\right\}^{\frac1p}<\fz$$
with the usual modification made in case of $q=\fz$.
\eli
\end{definition}

\begin{remark}\label{Rem-Ftau}
	These spaces were introduced in \cite{ysy} and proved therein to be quasi-Banach spaces. In the Banach case  the scale of Nikol'skij-Besov type spaces ${\bt(\rn)}$ had already been introduced and investigated in \cite{ElBaraka1,ElBaraka2, ElBaraka3}.  It is easy to see that, when $\tau=0$, then $\bt(\rn)$ and $\ft(\rn)$
	coincide with the classical
	Besov space $\B(\rn)$ and Triebel-Lizorkin space $\F(\rn)$,
	respectively. In case of $\tau<0$ the spaces are trivial, 
	$\bt(\rn)=\ft(\rn)=\{0\}$, $\tau<0$.  There exists extensive literature on such spaces; we
	refer, in particular, to the series of monographs \cite{T-F1,T-F2,t06,T20} for a	comprehensive treatment.
\end{remark}

\noindent{\em Convention.}~We adopt the nowadays usual custom to write $\A(\rn)$ instead of $\B(\rn)$ or $\F(\rn)$, and $\at(\rn)$ instead of $\bt(\rn)$ or $\ft(\rn)$, respectively, when both scales of spaces are meant simultaneously in some context.

We have elementary embeddings within this scale of spaces (see \cite[Proposition~2.1]{ysy}),
\begin{equation} \label{elem-0-t}
A^{s+\ve,\tau}_{p,r}(\rn) \hookrightarrow \at(\rn) \qquad\text{if}\quad \varepsilon\in(0,\fz), \quad r,\,q\in(0,\infty],
\end{equation}
and
\begin{equation} \label{elem-1-t}
{A}^{s,\tau}_{p,q_1}(\rn)  \hookrightarrow {A}^{s,\tau}_{p,q_2}(\rn)\quad\text{if} \quad q_1\le q_2,
\end{equation}
as well as
\begin{equation}\label{elem-tau}
B^{s,\tau}_{p,\min\{p,q\}}(\rn)\, \hookrightarrow \, \ft(\rn)\, \hookrightarrow \, B^{s,\tau}_{p,\max\{p,q\}}(\rn),
\end{equation}
which directly extends the well-known classical case from $\tau=0$ to $\tau\in [0,\infty)$, $p\in(0,\fz)$, $q\in(0,\fz]$ and $s\in\real$.

It is also known from \cite[Proposition~2.6]{ysy}
that 
\begin{equation} \label{010319}
A^{s,\tau}_{p,q}(\rn) \hookrightarrow B^{s+d(\tau-\frac1p)}_{\fz,\fz}(\rn). 
\end{equation} 
The following  remarkable feature was proved in \cite{yy02}.
\begin{proposition}\label{yy02}
	Let $s\in\rr$, $\tau\in[0,\fz)$  and $p,\,q\in(0,\fz]$ (with $p<\infty$ in the $F$-case). If
	either $\tau>\frac1p$ or $\tau=\frac1p$ and $q=\infty$, then $A^{s,\tau}_{p,q}(\rn) = B^{s+d(\tau-\frac1p)}_{\fz,\fz}(\rn)$.
\end{proposition}

Now we come to smoothness  spaces of Morrey type   $\MB(\rn)$ and $\MF(\rn)$. Recall first that the \emph{Morrey space}
$\M(\rn)$, $0<p\le u<\infty $, is defined to be the set of all
locally $p$-integrable functions $f\in L_p^{\mathrm{loc}}(\rn)$  such that
$$
\|f \mid {\M(\rn)}\| :=\, \sup_{x\in \rn, R>0} R^{\frac{d}{u}-\frac{d}{p}}
\left[\int_{B(x,R)} |f(y)|^p \dint y \right]^{\frac{1}{p}}\, <\, \infty\, .
$$

\begin{remark}
	The spaces $\M(\rn)$ are quasi-Banach spaces (Banach spaces for $p \ge 1$).
	They originated from Morrey's study on PDE (see \cite{Mor}) and are part of the wider class of Morrey-Campanato spaces; cf. \cite{Pee}. They can be considered as a complement to $L_p$ spaces. As a matter of fact, $\cM_{p,p}(\rn) = L_p(\rn)$ with $p\in(0,\infty)$.
	To extend this relation, we put  $\cM_{\infty,\infty}(\rn)  = L_\infty(\rn)$. One can easily see that $\M(\rn)=\{0\}$ for $u<p$, and that for  $0<p_2 \le p_1 \le u < \infty$,
	\begin{equation} \label{LinM}
	L_u(\rn)= \cM_{u,u}(\rn) \hookrightarrow  \cM_{u,p_1}(\rn)\hookrightarrow  \cM_{u,p_2}(\rn).
	\end{equation}
	In an analogous way, one can define the spaces $\cM_{\infty,p}(\rn)$, $p\in(0, \infty)$, but using the Lebesgue differentiation theorem, one can easily prove  that
	$\cM_{\infty, p}(\rn) = L_\infty(\rn)$.
\end{remark}

\begin{definition}\label{d2.5}
	Let $0 <p\leq  u<\infty$ or $p=u=\infty$. Let  $q\in(0,\infty]$, $s\in \real$ and $\vz_0$, $\vz\in\cs(\rn)$
	be as in \eqref{e1.0} and \eqref{e1.1}, respectively.
	\bli
	\item[{\bfseries\upshape (i)}]
	The  {\em Besov-Morrey   space}
	$\MB(\rn)$ is defined to be the set of all distributions $f\in \SpRn$ such that
	\begin{align}\label{BM}
	\big\|f\mid \MB(\rn)\big\|=
	\bigg[\sum_{j=0}^{\infty}2^{jsq}\big\| \varphi_j \ast f\mid
	\M(\rn)\big\|^q \bigg]^{1/q} < \infty
	\end{align}
	with the usual modification made in case of $q=\fz$.
	\item[{\bfseries\upshape  (ii)}]
	Let $u\in(0,\fz)$. The  {\em Triebel-Lizorkin-Morrey  space} $\MF(\rn)$
	is defined to be the set of all distributions $f\in   \SpRn$ such that
	\begin{align}\label{FM}
	\big\|f \mid \MF(\rn)\big\|=\bigg\|\bigg[\sum_{j=0}^{\infty}2^{jsq} |
	(\varphi_j\ast f)(\cdot)|^q\bigg]^{1/q}
	\mid \M(\rn)\bigg\| <\infty
	\end{align}
	with the usual modification made in case of  $q=\fz$.
	\eli
\end{definition}

\noindent{\em Convention.} Again we adopt the usual custom to write $\MA$ instead of $\MB$ or $\MF$, when both scales of spaces are meant simultaneously in some context. 

\begin{remark}
	The  spaces $\MA(\rn)$ are
	independent of the particular choices of $\vz_0$, $\vz$ appearing in their definitions.
	They are quasi-Banach spaces
	(Banach spaces for $p,\,q\geq 1$), and $\mathcal{S}(\rn) \hookrightarrow
	\MA(\rn)\hookrightarrow \mathcal{S}'(\rn)$.  Moreover, for $u=p$
	we re-obtain the usual  Besov and Triebel-Lizorkin spaces,
	\begin{equation}
	{\cal A}^{s}_{p,p,q}(\rn) = \A(\rn) = A^{s,0}_{p,q}(\rn). \label{MB=B}
	\end{equation}
	Besov-Morrey spaces were introduced by Kozono and Yamazaki in
	\cite{KY}. They studied semi-linear heat equations and Navier-Stokes
	equations with initial data belonging to  Besov-Morrey spaces.  The
	investigations were continued by Mazzucato \cite{Maz}, where one can find the
	atomic decomposition of some spaces. The Triebel-Lizorkin-Morrey spaces
	were later introduced by  Tang and Xu \cite{TX}. We follow the
	ideas of Tang and Xu \cite{TX}, where a somewhat  different definition is proposed. The ideas were further developed by Sawano and Tanaka \cite{ST1,ST2,Saw1,Saw2}. The most systematic and general approach to the spaces of this type  can  be found in the monograph \cite{ysy} or in the  survey papers by Sickel \cite{s011,s011a}, which we also recommend for further up-to-date references on this subject. We refer to the recent monographs \cite{FHS-MS-1,FHS-MS-2} for applications.
\end{remark}

It turned out that many of the results from the classical situation have their  counterparts for the spaces $\mathcal{A}^s_{u,p,q}(\rn)$, e.\,g.,
\begin{equation} \label{elem-0}
{\mathcal A}^{s+\varepsilon}_{u,p,r}(\rn)  \hookrightarrow
\MA(\rn)\qquad\text{if}\quad \varepsilon>0, \quad r\in(0,\infty],
\end{equation}
and ${\cal A}^{s}_{u,p,q_1}(\rn)  \hookrightarrow {\cal A}^{s}_{u,p,q_2}(\rn)$ if $q_1\le q_2$. However, there also exist some differences.
Sawano proved in \cite{Saw2} that, for $s\in\real$ and $0<p< u<\infty$,
\begin{equation}\label{elem}
{\cal N}^s_{u,p,\min\{p,q\}}(\rn)\, \hookrightarrow \, \MF(\rn)\, \hookrightarrow \,{\cal N}^s_{u,p,\infty}(\rn),
\end{equation}
where, for the latter embedding, $r=\infty$ cannot be improved -- unlike
in case of $u=p$ (see \eqref{elem-tau} with $\tau=0$). More precisely,
\[
\MF(\rn)\hookrightarrow {\mathcal N}^s_{u,p,r}(\rn)\quad\text{if, and only if,}\quad r=\infty ~~ \text{or} ~~  u=p\ \text{and}\ r\ge \max\{p,\,q\}.
\]
On the other hand, Mazzucato has shown in \cite[Proposition~4.1]{Maz} that
\[
\mathcal{E}^0_{u,p,2}(\rn)=\M(\rn),\quad 1<p\leq u<\infty,
\]
in particular,
\begin{equation}\label{E-Lp}
\mathcal{E}^0_{p,p,2}(\rn)=L_p(\rn)=F^0_{p,2}(\rn),\quad p\in(1,\infty).
\end{equation}

\begin{remark}\label{N-Bt-spaces}
	Let $s$, $u$, $p$ and $q$ be as in Definition~\ref{d2.5}
	and $\tau\in[0,\fz)$.
	It is known  that   
	\begin{equation}
	\MB(\rn) \hookrightarrow  \bt(\rn) \qquad \text{with}\qquad \tau={1}/{p}- {1}/{u}, 
	\label{N-BT-emb}
	\end{equation}
    cf. \cite[Corollary~3.3]{ysy}.
	Moreover, the above embedding is proper if $\tau>0$ and $q<\infty$. If $\tau=0$ or $q=\infty$, then both spaces coincide with each other, in particular,
	\begin{equation}
	\mathcal{N}^{s}_{u,p,\infty}(\rn)  =  B^{s,\frac{1}{p}- \frac{1}{u}}_{p,\infty}(\rn).
	\label{N-BT-equal}
	\end{equation}
	As for the $F$-spaces, if $0\le \tau <{1}/{p}$,
	then
	\begin{equation}\label{fte}
	\ft(\rn)\, = \, \MF(\rn)\quad\text{with }\quad \tau =
	{1}/{p}-{1}/{u}\, ,\quad 0 < p\le u < \infty\, ;
	\end{equation}
	cf. \cite[Corollary~3.3]{ysy}. Moreover, if $p\in(0,\infty)$ and $q\in (0,\infty)$, then
	\begin{equation}\label{ftbt}
	F^{s,\, \frac{1}{p} }_{p\, ,\,q}(\rn) \, = \, F^{s}_{\infty,\,q}(\rn)\, = \, B^{s,\, \frac1q }_{q\, ,\,q}(\rn) \, ;
	\end{equation}
	cf. \cite[Propositions~3.4,~3.5]{s011} and \cite[Remark~10]{s011a}.\\
\end{remark}

\begin{remark}\label{bmo-def}
Recall that the space $\bmo(\rn)$ is covered by the above scale. More precisely, consider the local (non-homogeneous) space of functions of bounded mean oscillation, $\bmo(\rn)$, consisting of all locally integrable
functions $\ f\in \Lloc(\rn) $ satisfying that
\begin{equation*}
 \left\| f \right\|_{\bmo}:=
\sup_{|Q|\leq 1}\; \frac{1}{|Q|} \int\limits_Q |f(x)-f_Q| \dint x + \sup_{|Q|>
1}\; \frac{1}{|Q|} \int\limits_Q |f(x)| \dint x<\infty,
\end{equation*}
where $ Q $ appearing in the above definition runs over all cubes in $\rn$, and $ f_Q $ denotes the mean value of $ f $ with
respect to $ Q$, namely, $ f_Q := \frac{1}{|Q|} \;\int_Q f(x)\dint x$,
cf. \cite[2.2.2(viii)]{T-F1}. The space $\bmo(\rn)$ coincides with $F^{0}_{\infty, 2}(\rn)$,  cf. \cite[Thm.~2.5.8/2]{T-F1}. 
Hence the above result \eqref{ftbt} implies, in particular,
\begin{equation}\label{ft=bmo}
\bmo(\rn)= F^{0}_{\infty,2}(\rn)= F^{0, 1/p}_{p, 2}(\rn)= {B^{0, 1/2}_{2, 2}(\rn)}, \quad 0<p<\infty.
\end{equation}
\end{remark}

\begin{remark}\label{T-hybrid}
In contrast to this approach, Triebel followed the original Morrey-Campanato ideas to develop local spaces $\mathcal{L}^r\A(\rn)$ in \cite{t13}, and so-called `hybrid' spaces $L^r\A(\rn)$ in \cite{t14}, where $0<p<\infty$, $0<q\leq\infty$, $s\in\real$, and $-\frac{d}{p}\leq r<\infty$. This construction is based on wavelet decompositions and also combines local and global elements as in Definitions~\ref{d1} and \ref{d2.5}. However, Triebel proved in \cite[Thm.~3.38]{t14} that
\begin{equation} \label{hybrid=tau}
L^r\A(\rn) = \at(\rn), \qquad \tau=\frac1p+\frac{r}{d},
\end{equation}
in all admitted cases. We return to this coincidence below.
\end{remark}

As mentioned previously, in this paper we are interested in studying embeddings of type 
\begin{equation*}
\id_\tau : \ate \hookrightarrow \atz
\end{equation*}
on domains. To do so, we will strongly rely on the corresponding results for these spaces on $\rn$, obtained in \cite{YHSY}. In order to make the reading easier, we recall those results here. We begin with the situation of 
  Besov-type spaces where the results can be found in \cite[Theorems~2.4,~2.5]{YHSY}.

\begin{theorem}[\cite{YHSY}]\label{B-rn}
	Let  $s_i\in \real$, $0<q_i\leq\infty$, $0<p_i\leq \infty$ and $\tau_i\geq 0$, $i=1,2$. 
	\begin{itemize}
		\item[{\bfseries\upshape (i)}] Let $\tau_2 > \frac{1}{p_2}$ or $\tau_2=\frac{1}{p_2}$ and $q_2=\infty$. Then the embedding
		\begin{equation}\label{eq:B-rn}
		\bte(\rn) \hookrightarrow \btz(\rn)
		\end{equation}
		holds if, and only if, $\quad \displaystyle\frac{s_1-s_2}{d}\geq \frac{1}{p_1}- \tau_1 - \frac{1}{p_2}+\tau_2$. 
		\item[{\bfseries\upshape (ii)}] Let $\tau_1 > \frac{1}{p_1}$ or $\tau_1=\frac{1}{p_1}$ and $q_1=\infty$. Then the embedding \eqref{eq:B-rn} holds if, and only if,
		\begin{align*}
		&\qquad \frac{s_1-s_2}{d}> \frac{1}{p_1}- \tau_1 - \frac{1}{p_2}+\tau_2 \quad \mbox{and} \quad \tau_2 \geq \frac{1}{p_2}\\
		\mbox{or} & \qquad \frac{s_1-s_2}{d}= \frac{1}{p_1}- \tau_1 - \frac{1}{p_2}+\tau_2 \quad \mbox{and} \quad \tau_2 > \frac{1}{p_2}\\
		\mbox{or} & \qquad  \frac{s_1-s_2}{d}= \frac{1}{p_1}- \tau_1 - \frac{1}{p_2}+\tau_2 \quad \mbox{and} \quad \tau_2 = \frac{1}{p_2} \quad \mbox{and} \quad q_2=\infty.
		\end{align*}
		\item[{\bfseries\upshape (iii)}] Assume that  $\tau_i < \frac{1}{p_i}$ or $\tau_i=\frac{1}{p_i}$ and $q_i<\infty$, $i=1,2$. 
		\begin{itemize}
			\item[{\upshape (a)}] The embedding \eqref{eq:B-rn}  holds true if
			$$ \frac{1}{p_1}- \tau_1 - \frac{1}{p_2}+\tau_2\geq 0, \qquad \frac{\tau_1}{p_2}\leq \frac{\tau_2}{p_1}$$
			and
			\begin{align*}
			&\qquad \frac{s_1-s_2}{d}> \frac{1}{p_1}- \tau_1 - \frac{1}{p_2}+\tau_2\\
			\mbox{or} & \qquad\frac{s_1-s_2}{d}= \frac{1}{p_1}- \tau_1 - \frac{1}{p_2}+\tau_2 \\
			&\qquad \mbox{with} \\
			&\qquad \qquad \quad (s_1-s_2)(\tau_1-\tau_2)\neq 0, \quad \frac{\tau_1}{p_2}< \frac{\tau_2}{p_1}\\
			&\qquad  \mbox{or} \qquad (s_1-s_2)(\tau_1-\tau_2)\neq 0, \quad \frac{\tau_1}{p_2}= \frac{\tau_2}{p_1}, \quad \frac{\tau_1}{q_2}\leq \frac{\tau_2}{q_1}\\
			&\qquad  \mbox{or} \qquad (s_1-s_2)(\tau_1-\tau_2)=0, \quad q_1 \leq q_2,\\
			&\qquad \qquad \quad \mbox{and} \quad p_1\geq p_2 \quad \mbox{if} \quad s_1=s_2 \quad \mbox{and} \quad \tau_1 p_1 = \tau_2 p_2 =1. 
			\end{align*}
			\item[{\upshape (b)}] The conditions  $ \frac{1}{p_1}- \tau_1 - \frac{1}{p_2}+\tau_2\geq 0$ and $\frac{\tau_1}{p_2}\leq \frac{\tau_2}{p_1}$, and $s_1-s_2\geq \frac{d}{p_1}- d\tau_1 - \frac{d}{p_2}+d\tau_2$ as well as $q_1 \leq q_2$ if $s_1=s_2$ are also necessary for the embedding \eqref{eq:B-rn}. 
		\end{itemize}			
	\end{itemize}
\end{theorem} 

The counterpart for $F$-spaces reads as follows, we refer to \cite[Corollaries~5.8,~5.9]{YHSY} for details.

\begin{theorem}[\cite{YHSY}]\label{F-rn}
	Let  $s_i\in \real$, $0<q_i\leq\infty$, $0<p_i< \infty$ and  $\tau_i\geq 0$, $i=1,2$. 
	\begin{itemize}
		\item[{\bfseries\upshape (i)}] Let $\tau_i > \frac{1}{p_i}$ or $\tau_i=\frac{1}{p_i}$ and $q_i=\infty$, $i=1,2$. Then the embedding
		\begin{equation}\label{eq:F-rn}
		\fte(\rn) \hookrightarrow \ftz(\rn)
		\end{equation}
		holds if, and only if, $\quad \displaystyle\frac{s_1-s_2}{d}\geq \frac{1}{p_1}- \tau_1 - \frac{1}{p_2}+\tau_2$. 
		\item[{\bfseries\upshape (ii)}] Assume that  $\tau_i < \frac{1}{p_i}$, $i=1,2$. Then the embedding \eqref{eq:F-rn} holds if, and only if, 
		$$ \frac{1}{p_1}- \tau_1 - \frac{1}{p_2}+\tau_2\geq 0, \qquad \frac{\tau_1}{p_2}\leq \frac{\tau_2}{p_1}$$
		and
		\begin{align*}
		&\qquad \frac{s_1-s_2}{d}>\frac{1}{p_1} - \tau_1 - \frac{1}{p_2}+\tau_2\\\
		\mbox{or} & \qquad\frac{s_1-s_2}{d}= \frac{1}{p_1}- \tau_1 - \frac{1}{p_2}+\tau_2 \\ 
		\mbox{or} & \qquad s_1=s_2 \quad \mbox{and} \quad q_1 \leq q_2. 
		\end{align*}
	\end{itemize}
\end{theorem}

\subsection{Spaces on domains}\label{def-domain}
Let $\Omega$ denote an open, nontrivial subset of $\rn$. We consider smoothness Morrey spaces on $\Omega$ defined by restriction. Let ${\cal D}(\Omega)$ be the set of all infinitely differentiable functions supported in $\Omega$ and denote by ${\cal D}'(\Omega)$ its dual. 
Since we are able to define the extension operator $\ext: {\cal D}(\Omega) \rightarrow  \SRn$, cf.
\cite{Saw2010}, the restriction operator $\re :   \SpRn \rightarrow {\cal D}'(\Omega)$ can be defined naturally as an adjoint operator
\[
\langle \re  (f), \varphi\rangle= \langle f, \ext (\varphi)\rangle, \quad f\in\SpRn, 
\]
where $\varphi\in \mathcal{D}(\Omega)$. We will write $f\vert_{\Omega}={\rm re } (f)$.



\begin{definition}\label{tau-spaces-Omega}
	Let $s\in\real$, $\tau\in [0,\infty)$, $q\in (0,\infty]$ and $p\in (0,\infty]$ (with $p<\infty$ in the case of $\at=\ft$). Then 
	$\at(\Omega)$ is defined by
	\[
	\at(\Omega):=\big\{f\in {\cal D}'(\Omega): f=g\vert_{\Omega} \text{ for some } g\in \at(\rn)\big\}
	\]
	endowed with the quasi-norm
	\[
	\big\|f\mid \at(\Omega)\big\|:= \inf \big\{ \|g\mid \at(\rn)\|:  f=g\vert_{\Omega}, \; g\in  \at(\rn)\big\}.
	\]
\end{definition}

\begin{remark}\label{tau-onOmega}
	The spaces $\at(\Omega)$ are quasi-Banach spaces (Banach spaces for $p,q\geq 1$).  When $\tau=0$ we re-obtain the usual Besov and Triebel-Lizorkin spaces defined on domains. For the particular case of $\Omega$ being a bounded $C^{\infty}$ domain in $\rn$, some properties were studied in \cite[Section~6.4.2]{ysy}. In particular, according to \cite[Theorem~6.13]{ysy}, for such a domain $\Omega$, there exists a linear and bounded extension operator
	\begin{equation}\label{ext-tau-1}
	\ext: \at(\Omega)\to \at(\rn),\quad \text{where}\quad 1\leq p<\infty, 0<q\leq\infty, s\in\real, \tau\geq 0,
	\end{equation}  
	such that 
	\begin{equation}\label{ext-tau-2}
	\mathrm{re} \circ \ext = \id\quad\text{in}\quad \at(\Omega),
	\end{equation}  
	where $\mathrm{re}: \at(\rn)\to \at(\Omega)$ is the restriction operator as above. 
	
	Moreover, in \cite{hms} we studied the question under what assumptions these spaces consist of regular distributions only.  \ignore{In \cite{YHMSY} we considered the approximation numbers of some special compact embedding of $\at(\Omega)$ into $L_\infty(\Omega)$.} 
\end{remark}

\begin{remark}\label{emb-Omega}
  Let us mention that we have the counterparts of many continuous embeddings stated in the previous subsection for spaces on $\rn$ when dealing with spaces restricted to bounded domains. 
  We recall them in further detail if appropriate and necessary for our arguments below.
\end{remark}

Later we shall mainly deal with Lipschitz domains. Therefore we recall the concept for convenience.
{By a Lipschitz domain we mean either a special or bounded Lipschitz domain. A special Lipschitz domain is defined as an open set $\Omega \subset \rn$ lying above a graph of a Lipschitz function $\omega: \real^{d-1} \rightarrow \real$. More precisely, 
$$
\Omega = \{(x', x_d) \in \rn : x_d > \omega(x') \},
$$
where
\begin{equation}\label{lip1}
|\omega(x')-\omega(y')|\leq A \, |x'-y'|, \quad x', y' \in \real^{d-1}. 
\end{equation}
  A bounded Lipschitz domain is a bounded domain $\Omega$ whose boundary $\partial \Omega$ can be covered by a finite number of open balls $B_k$ so that, possibly after an appropriate rotation, $\partial \Omega \cap B_k$ for each $k$ is a part of the graph of a Lipschitz function.
Let $\Omega$ be a special Lipschitz domain defined by a Lipschitz function $\omega$ that satisfies the condition \eqref{lip1}.  We put
$$
K:= \{ (x', x_n) \in \rn: |x'|<A^{-1} x_n\}
$$
and $-K :=\{-x: x\in K\}$.  Then $K$ has the property that $x+K \subset \Omega$ for any $x \in \Omega$. Moreover let $\mathcal{S}'(\Omega)$ denote the subspace of $\mathcal{D}'(\Omega)$ consisting of all distributions of finite order and of at most polynomial growth at infinity, that is, $f\in  \mathcal{S}'(\Omega)$ if, and only if, the estimate 
\[
|\langle f,\eta\rangle| \le c \sup_{x\in  \Omega, |\alpha|\le M}|\Dd^\alpha \eta(x)|(1+|x|)^M,\qquad \eta\in \mathcal{D}(\Omega), 
\] 
holds with some constants $c$ and $M$ depending on $f$.
}


\begin{remark}\label{R-lim-B}
  As already mentioned, we shall study continuous embeddings of function spaces on domains. For convenience let us recall what is well-known for the {classical} function spaces 
  $\A(\Omega)$. Let $s_i\in\real$, $0<p_i, q_i\leq\infty$, $i=1,2$, and $\Omega\subset\rn$ a bounded Lipschitz domain. Then
  \[
  \id^B_\Omega : \Be(\Omega) \hookrightarrow \Bz(\Omega)
  \]
  is continuous if, and only if, 
  \begin{align}\label{ddh-1}
    \text{either}\qquad & \frac{s_1-s_2}{d}  > \max\left\{\frac{1}{p_1}-\frac{1}{p_2}, 0\right\}\\
\nonumber
    \text{or}\qquad & \frac{s_1-s_2}{d}  = \max\left\{\frac{1}{p_1}-\frac{1}{p_2}, 0\right\}\quad\text{and}\quad q_1\leq q_2.
  \end{align}
Assume, in addition, $p_i<\infty$, $i=1,2$. Then
   \[
  \id^F_\Omega : \Fe(\Omega) \hookrightarrow \Fz(\Omega)
  \]
  is continuous if, and only if, either \eqref{ddh-1} is satisfied,
  \begin{align*}
  \text{or}\qquad & \frac{s_1-s_2}{d}  =  \max\left\{\frac{1}{p_1}-\frac{1}{p_2},0\right\}=\frac{1}{p_1}-\frac{1}{p_2}>0,\\
    \text{or}\qquad & \frac{s_1-s_2}{d}  = \max\left\{\frac{1}{p_1}-\frac{1}{p_2}, 0\right\}=0 \quad\text{and}\quad q_1\leq q_2.
  \end{align*}
For (partial) results we refer to \cite[Sect.~2.5.1]{ET}, \cite[p.~60]{t06}, and, quite recently, the extension to spaces $F^s_\infty,q(\Omega)$ in \cite[Sect.~2.6.5]{T20}; cf. also our results in \cite[Thm.~3.1]{hs12b} and  \cite[Thm.~5.2]{hs14}. However, in most of the above cases, the domain $\Omega$ is there assumed to be smooth.
\end{remark}

\subsection{Wavelet decomposition in Besov-type and Triebel-Lizorkin-type spaces}


\def\hs{\hspace{0.3cm}}
\def\wz{\widetilde}

In this section we describe the wavelet characterisation of Besov-type and Triebel-Lizorkin-type spaces proved in \cite{lsuyy}. This is a key tool when studying embedding properties of function spaces, since it allows one to transfer the problem to the corresponding sequence spaces. 

Let $\wz{\phi}$ be a scaling function on $\real$ with compact support and of sufficiently high regularity, and $\wz{\psi}$ the corresponding orthonormal wavelet. Then the tensor product ansatz yields a scaling function $\phi$ and associated wavelets $\psi_1, \dots, \psi_{2^d-1}$, all defined on $\rn$; see, e.g. \cite[Proposition~5.2]{wo97}. We suppose that 
\[
\phi, \psi_i \in C^{N_1}(\rn) \qquad \mbox{and} \qquad \supp \phi, \supp \psi_i \subset [-N_2, N_2]^d,\quad i=1, \dots, 2^d-1,
\]
for some $N_1, N_2 \in \nn$.

For $k\in\zn$, $j\in\nn_0$ and $i\in\{1,\ldots,2^d-1\}$, define
\[
\phi_{j,k}(x):= 2^{jd/2}\phi(2^jx-k)\hs \quad \text{and} \quad \hs
\psi_{i,j,k}(x):= 2^{jd/2}\psi_i(2^jx-k),\hs x\in\rn.
\]
It is well known that
\[
\int\limits_\rn \psi_{i,j,k}(x)\, x^\gamma\dint x = 0 \qquad  \mbox{if}
\qquad  |\gamma|\le N_1
\]
(see \cite[Proposition~3.1]{wo97}), and
\begin{equation*}
\{\phi_{0,k}: \ k\in\zn\}\: \cup \: \{\psi_{i,j,k}:\ k\in\zn,\
j\in\nn_0,\ i\in\{1,\ldots,2^d-1\}\}
\end{equation*}
forms an {\it orthonormal basis} of $L_2(\rn)$ (see, e.\,g., \cite[Section~3.9]{me} or \cite[Section~3.1]{t06}).
Hence
\begin{equation}\label{wavelet}
f=\sum_{k\in\zn}\, \lambda_k \, \phi_{0,k}+\sum_{i=1}^{2^d-1} \sum_{j=0}^\infty
\sum_{k\in\zn}\, \lambda_{i,j,k}\, \psi_{i,j,k}\,
\end{equation}
in $\L_2(\rn)$, where  $\lambda_k := \langle f,\,\phi_{0,k}\rangle$ and
$\lambda_{i,j,k}:= \langle f,\,\psi_{i,j,k}\rangle$. {We will denote by $\lambda(f)$ the following sequence:
\[\lambda(f) := \left(\lambda_k,\lambda_{i,j,k}\right)= \big( \langle f,\,\phi_{0,k}\rangle, \langle f,\,\psi_{i,j,k}\rangle \big).
\]}

\begin{definition}\label{dts}
	Let $s\in \rr$, $\tau\in[0,\fz)$ and $q\in(0,\fz]$.
\bli
      \item[{\bfseries\upshape (i)}] Let $p \in (0, \fz]$. The \emph{sequence space}
	${\sbt}:={\sbt}(\rn)$ is defined to be the space of all complex-valued
	sequences $t:=\{t_{i,j,m}:\ i\in\{1,\ldots,2^d-1\}, j\in\nn_0, m\in\zn\}$ such that $\|t\mid {\sbt}\|<\fz$, where
	$$\|t\mid {\sbt}\|:=
	\sup_{P\in\mathcal{Q}}\frac1{|P|^{\tau}}\left\{\sum_{j=\max\{j_P,0\}}^\fz
	2^{j(s+\frac d2-\frac dp)q} \sum_{i=1}^{2^d-1}
	\left[\sum_{m:\ Q_{j,m}\subset P}
	|t_{i,j,m}|^p\right]^{\frac qp}\right\}^{\frac 1q}, $$
	with the usual modification when $p=\infty$ or $q=\infty$.
	{\item[{\bfseries\upshape (ii)}] Let $p \in (0, \fz)$. The \emph{sequence space}
	${\sft}:={\sft}(\rn)$ is defined to be the space of all complex-valued
	sequences $t:=\{t_{i,j,m}:\ i\in\{1,\ldots,2^d-1\}, j\in\nn_0, m\in\zn\}$ such that $\|t\mid {\sft}\|<\fz$, where
	$$\|t\mid {\sft}\|:=
	\sup_{P\in\mathcal{Q}}\frac1{|P|^{\tau}}\left\{\int_P \left[\sum_{j=\max\{j_P,0\}}^\fz
	2^{j(s+\frac d2)q} \sum_{i=1}^{2^d-1}
	\sum_{m\in \zn}
	|t_{i,j,m}|^q \, \chi_{Q_{j,m}}(x)\right]^{\frac pq} \dint x \right\}^{\frac 1p}, $$
	with the usual modification when $q=\infty$.}	
\eli
\end{definition}

As a special case of  \cite[Theorem~4.12]{lsuyy}, we have the following
wavelet characterisation of $\at(\rn)$.

\begin{proposition}[\cite{lsuyy}]\label{wav-type2}
	Let $s\in\rr$, $\tau\in[0,\fz)$, $q\in(0,\fz]$. Moreover, let $N_1 \in \nn_0$ be such that, when $A=B$ and $p\in (0,\infty]$,
	\begin{equation*}
	N_1+1>\max\left\{d+\frac dp-d\tau-s,  \frac{2 d}{\min\{p,1\}}+d\tau+1, d+\frac dp+\frac d2,
	d+s, \frac dp -s, {s+d\tau}\right\},
	\end{equation*}
	and when $A=F$ and $p\in (0,\infty)$, 
	\begin{equation*}
	{N_1+1>\max\left\{d+\frac dp-d\tau-s,  \frac{2 d}{\min\{p,{q},1\}}+d\tau+1, d+\frac dp+\frac d2,
	d+s, \frac dp -s,   {s+d\tau}\right\}.}
	\end{equation*}
	Let $f\in\cs'(\rn)$. Then
	$f\in \at(\rn)$ if, and only if, $f$ can be represented as \eqref{wavelet} in $\cs'(\rn)$
	and
	\[
	\|\lambda(f)\mid \sat\|^\ast:=
	\sup_{P\in\mathcal{Q}}\frac1{|P|^{\tau}}
	\left\{\sum_{m:\ Q_{0,m}\subset P}
	| \langle f,\,\phi_{0,m}\rangle|^p\right\}^{\frac 1p} +
	\|\, \{\langle f,\,\psi_{i,j,m}\rangle\}_{i,j,m}\mid \sat\| <\infty\, .
	\]
	Moreover, $\|f \mid \at(\rn)\|$ is equivalent to
	$\|\lambda(f)\mid \sat\|^\ast.$
\end{proposition}

\section{Extension Operator}\label{extension}
As mentioned in Remark~\ref{tau-onOmega} above, if $\Omega$ is a bounded $C^{\infty}$ domain in $\rn$, an extension theorem for the spaces $\at(\Omega)$ was stated in \cite{ysy}, but with the assumption that $p \in [1, \infty )$. It is our aim in this section to establish {an extended} 
result, which holds for all $p \in (0, \infty]$ ($p\in (0,\infty)$ in the $F$-case). Similarly to what was done in \cite[Proposition~4.13]{YSY-15} for Besov-Morrey spaces $\MB$, we will follow Rychkov \cite[Theorem~2.2]{Ryc} in the construction of such an operator. We start with some preparation.\\

Let $q \in (0, \infty]$ and $\tau \in [0, \infty)$. Denote by $\mixB$ the set of all sequences $\{g_j\}_{j \in \no}$ of measurable functions on $\rn$ such that
\beq \label{mixB}
\|\{g_j\}_{j \in \no} \mid \mixB\| := \sup_{P \in \mathcal{Q}} \frac{1}{|P|^\tau} \left\{ \sum_{j = \max\{j_P, 0\}}^{\infty} \|g_j 	\mid L_p(P)\|^q\right\}^{1/q} < \infty.
\eeq
Similarly, $\mixF$ with $p\in (0, \infty)$ denotes the set of all sequences $\{g_j\}_{j \in \no}$ of measurable functions on $\rn$ such that
\beq\label{mixF}
\|\{g_j\}_{j \in \no}\mid \mixF\| := \sup_{P \in \mathcal{Q}} \frac{1}{|P|^\tau} \left\| \left\{ \sum_{j = \max\{j_P, 0\}}^{\infty} |g_j|^q\right\}^{1/q}  \mid L_p(P)\right\| < \infty.
\eeq

The first auxiliary lemma can be seen as a particular case of \cite[Lemma~2.9]{lsuyy}, where the authors considered more general {quasi-}norms. 

\begin{lemma} \label{lemma:aux1}Let $q\in (0, \infty ]$, $\tau \in [0, \infty)$. Let $D_1, D_2 \in (0, \infty)$ with $D_2 >d \tau$. For any sequence $\{g_\nu\}_{\nu \in \no}$ of measurable functions on $\rn$, consider 
$$G_j(x):= \sum_{\nu =0}^j 2^{-(j -\nu)D_2} g_\nu(x) + \sum_{\nu=j+1}^\infty2^{-(\nu-j)D_1}g_\nu(x), \quad j \in \no, \quad x \in \rn.$$
Then there exists a positive constant $C$, independent of $\{g_\nu\}_{\nu \in \no}$, such that, for all $p \in (0, \infty ]$,
$$
\|\{ G_j\}_{j \in \no} \mid \mixB\| \leq C\, \|\{g_\nu\}_{\nu \in \no} \mid \mixB\|,
$$
and, for all $p \in (0, \infty )$,
$$
\|\{ G_j\}_{j \in \no}\mid \mixF\| \leq C\, \|\{g_\nu\}_{\nu \in \no}\mid \mixF\|. 
$$
\end{lemma}
\vspace{0.5cm}
Let $(\vz_0,\vz)$ be an admissible pair and $f \in \SpRn$. For all $j \in \no$, $a \in (0, \infty)$ and $x \in \rn$, we denote
$$
\varphi_j^{\ast, a}f(x) := \sup_{y \in \rn} \frac{|\varphi_j \ast f(y)|}{(1+2^j|x-y|)^a}. 
$$

Another tool we will need later is the characterisation of the spaces via Peetre maximal functions. For the homogeneous version of the spaces, this result was proved in \cite[Theorem~1.1]{YY}. Here we use the results proved in \cite[Theorem~5.1]{YZY15-B} and \cite[Theorem~3.11]{YZY15-F} for the more general scale of Besov-type spaces and Triebel-Lizorkin-type spaces with variable exponents, respectively, which in particular cover our spaces. Adapted to our case, the results read as follows.

\begin{theorem} \label{Peetre} Let $s \in \real$, $\tau \in [0, \infty)$, $q \in (0, \infty]$ and $(\vz_0,\vz)$ be an admissible pair. 
\begin{itemize}
\item[{\bfseries\upshape (i)}] Let $p\in(0, \infty]$ and $a > \displaystyle \frac{d}{p}+d\tau$. Then 
$$
  \|f \mid {\bt(\rn)}\| 
  \sim \| \{2^{js}\varphi_j^{\ast, a}f\}_{j \in \no}\mid \mixB\|.
$$
\item[{\bfseries\upshape (ii)}] Let  $p\in(0, \infty)$ and $a > \displaystyle \frac{d}{\min\{p,q\}}+d\tau$. Then 
$$
  \|f \mid {\ft(\rn)}\| 
  \sim \| \{2^{js}\varphi_j^{\ast, a}f\}_{j \in \no}\mid \mixF\|.
$$
\end{itemize}
\end{theorem}

\begin{remark}\label{adm-pair} Note that in the above theorem the sequence $\{\varphi_j\}_{j \in \no}$ is built upon an admissible pair $(\varphi_0, \varphi)$, as in the Definition~\ref{d1}. Hence  
$\|f | {\bt(\rn)}\| \sim \|\{2^{js}(\varphi_j \ast f)\}_{j \in \no} \mid \mixB \|\ $ and $\|f | {\ft(\rn)}\| \sim \|\{2^{js}(\varphi_j \ast f)\}_{j \in \no}\mid \mixF\|$. 
  However, in {\cite{G20} and} \cite{GM18} the authors proved that, for {Besov-type and} Triebel-Lizorkin-type spaces with variable exponents, one can consider more general pairs of functions not only in this result, but also in the definition of the spaces.\\
  
\end{remark}

Following Rychkov \cite{Ryc}, we will prove the existence of a universal linear bounded extension operator if $\Omega$ is a Lipschitz domain, recall the explanations given at the end of Section~\ref{def-domain}. 
We first recall two results of Rychkov we shall need in our argument below.

  
\begin{lemma}[{\cite[Prop.~3.1]{Ryc}}]\label{ext_l1}
	Let $\Omega$ be a special Lipschitz domain. The distribution $f\in \mathcal{D}'(\Omega)$ belongs to  $\mathcal{S}'(\Omega)$ if and only if there exist  $g\in \mathcal{S}'(\rn)$ such that $f=g|_\Omega$. 
\end{lemma}

\begin{lemma}[{\cite[Thm.~4.1(a)]{Ryc}}]\label{ext_l2}
	Let $\Omega$ be a special Lipschitz domain and $K$ its associated cone. Let $-K :=\{-x: x\in K\}$ be a `reflected' cone. Then there exist functions  $\phi_0,\phi, \psi_0, \psi\in\mathcal{S}(\rn)$ supported in $-K$ such that
	\begin{align}\label{ext_l2_1}
		\int_{\rn}x^\alpha \phi(x) \dint x  &= 	\int_{\rn}x^\alpha \psi(x) \dint x = 0 \quad \text{for all multi-indices}\;\alpha, 
		\intertext{and}
		 	f &= \sum_{j \in \no} \psi_j \ast \phi_j \ast f \quad \mbox{in} \quad {\mathcal D}'(\Omega), \quad \text{for any} \; f\in \mathcal{S}'(\Omega), \label{ext_l2_2}
	\end{align} 
   where  $\phi_j(\cdot)= 2^{jd}\phi(2^j\cdot)$ and $\psi_j(\cdot)= 2^{jd}\psi(2^j\cdot)$, $j\in \nn$.  
\end{lemma}

{We can now state the main result of this section.}
	\begin{theorem}\label{Th:ext}
		Let $\Omega \subset \rn$ {be a special or bounded Lipschitz domain if $d\geq 2$, {or} an interval if $d=1$}. 
		Then there exists a linear bounded operator $\Ext$ which maps $\at(\Omega)$ into $\at(\rn)$ for all $s \in \real$, $\tau \in [0, \infty)$, $p \in (0, \infty]$ ($p<\infty$ in the $F$-case) and $q \in (0, \infty ]$, such that, for all $f \in {\mathcal D}'(\Omega)$, $\Ext f|_{\Omega}=f$ in ${\mathcal D}'(\Omega)$.
	\end{theorem}
	
	\begin{proof}
	  We concentrate now on Besov-type spaces with  $d\geq 2$ and give some details on the Triebel-Lizorkin scale at the end of the proof.   
	  We apply   the extension operator constructed by V.~Rychkov in \cite{Ryc} and follow the main  ideas of his proof. By a standard procedure (see \cite[Subsection~1.2]{Ryc}), we only need to consider the case when $\Omega$ is a special Lipschitz domain.

		 Let $\Omega$ be a special  Lipschitz domain. The spaces $\bt(\Omega)$ are defined by restriction therefore it follows from Lemma~\ref{ext_l1} and Lemma~\ref{ext_l2} that any distribution $f\in \bt(\Omega)$ can be represented in the form \eqref{ext_l2_2} with the functions $\phi$ and $\psi$ satisfying  \eqref{ext_l2_1}.

		 For any distribution $f \in {\mathcal S}'(\Omega)$, we define the mapping
		 \beq\label{eq:ext}
		 \ext f := \sum_{j \in \no} \psi_j \ast (\phi_j \ast f)_\Omega .
		 \eeq
		 Here we use the notation $g_\Omega$ to denote the extension of a function $g:\Omega \rightarrow \real$ from $\Omega$ to $\rn$ by setting
		 \beq
		 g_\Omega (x) = \begin{cases} g(x), & \mbox{if } x \in \Omega, \\ 0, & \mbox{if } x \in \rn \setminus \Omega. \end{cases} \nonumber
		 \eeq

	
		\emph{Step 1.} Let $\{g_j\}_{j \in \no}$ be a sequence of measurable functions. Moreover, let $\mathcal{M}_{N}^{g_j}(x)$ denote the Peetre maximal function of $g_j$,  namely 
		\beq
		\mathcal{M}_{N}^{g_j}(x):= \sup_{y \in \rn} \frac{|g_j(y)|}{(1+2^j|x-y|)^N} \nonumber
		\eeq
		for all $x \in \rn$ and $N \in \nn \cap \left( \frac{d}{\min\{1,p\}}, \infty\right)$.  We prove that if  $\{2^{js}\, \mathcal{M}_{N}^{g_j}\}_{j \in \no} \in \mixB $, then the series $\sum_{j \in \no} \psi_j \ast g_j$ converges in $\SpRn$ and the  $\bt(\rn)$ norm of its sum can be estimated by 
		$\|\{2^{js}\, \mathcal{M}_{N}^{g_j}\}_{j \in \no} \mid \mixB\|$.	
		
		We start with the following elementary inequality
		\begin{align}
			|\phi_\ell \ast \psi_j \ast g_j(x)| \le \mathcal{M}_{N}^{g_j}(x) 
			\int|\phi_\ell \ast \psi_j(y)|(1+2^j|y|)^N \dint y\; .  
	    \end{align}	
        Bui, Paluszy\'nski and Taibleson proved in \cite{BPT}
       that 
			\begin{align}
			\int|\phi_\ell \ast \psi_j(y)|(1+2^j|y|)^N \dint y \le C_{M,N} 2^{-|\ell-j|M}\qquad \text{for all} \; M>0, 	
		\end{align} 
	cf. \cite[Lemma 2.1]{BPT} or \cite{Ryc} the proof of Theorem 4.1. 
		We take $M>|s|-d\tau$ and put $\sigma=M- |s|$. Then 
		\beq
		2^{\ell s}|\phi_\ell \ast \psi_j \ast g_j(x)| \lesssim 2^{-|\ell-j|\sigma}2^{js} \mathcal{M}^{g_j}_{N}(x), \quad x \in \rn, \quad \ell \in \no. \nonumber
		\eeq
		If the  sequence $\{g_j\}_{j \in \no}$ is such that  $\|\{2^{js}\,\mathcal{M}^{g_j}_{N}\}_{j \in \no } \mid \mixB\| <\infty$, then there exists a constant $c$ such that for any dyadic cube $P$ we have  $\|\mathcal{M}^{g_j}_{N}|L_p(P)\| \le c |P|^\tau$. In consequence, any function $g_j$ is a tempered distribution and 
		$\psi_j \ast g_j\in \mathcal{S}'(\rn)$, $j \in \no$. 
		
		It holds true that 
		\begin{align}\label{eq:est0}
			\big\|\psi_j \ast g_j \mid B^{s-2\sigma, \tau}_{p,q}(\rn)\big\| &\lesssim \sup_{P \in\mathcal{Q}} \frac{1}{|P|^\tau} \left\{ \sum_{\ell=\max\{j_P, 0\}}^{\infty} 2^{-(\ell 2 \sigma + |\ell-j|\sigma)q}\right\}^{1/q} \bigg\|2^{js} \mathcal{M}_N^{g_j}\mid L_p(P)\bigg\| \nonumber\\
			&\lesssim 2^{-j \sigma}\, \left\|\{2^{ks}\, \mathcal{M}_N^{g_k}\}_{k \in \no}\mid \mixB\right\|,
		\end{align}
		using that $|\ell-j|\geq j-\ell$. Therefore $\sum_{j \in \no} \psi_j \ast g_j$ converges in $B^{s-2\sigma, \tau}_{p,q}(\rn)$ and hence in $\SpRn$, since $B^{s-2\sigma, \tau}_{p,q}(\rn) \hookrightarrow \SpRn$. In this way, we further have
		\beq
		2^{\ell s}\left| \phi_\ell \ast \left(\sum_{j=0}^\infty \psi_j \ast g_j \right) (x)\right| \lesssim \sum_{j=0}^\infty 2^{-|\ell-j|\sigma} 2^{js}\mathcal{M}_{N}^{g_j}(x), \quad x \in \rn, \quad \ell \in \no. \nonumber
		\eeq
		Applying this, we see that
		\begin{align}
			\bigg\| \sum_{j=0}^\infty \psi_j \ast g_j \mid \bt(\rn)\bigg\| &= \sup_{P \in \mathcal{Q}} \frac{1}{|P|^\tau} \left\{\sum_{\ell=\max\{j_P, 0\}}^{\infty} 2^{\ell sq} \left[ \int_P  \Big| \phi_\ell \ast \left(\sum_{j=0}^\infty \psi_j \ast g_j \right) (x)\Big|^p {\dint x} \right]^{q/p} \right\}^{1/q}\nonumber \\
			&\lesssim \sup_{P \in \mathcal{Q}} \frac{1}{|P|^\tau} \left\{\sum_{\ell=\max\{j_P, 0\}}^{\infty} \bigg\| \sum_{j=0}^\infty  2^{-|\ell-j|\sigma} \, 
			2^{js}\,\mathcal{M}^{g_j}_{N} \mid L_p(P)\bigg\|^{q} \right\}^{1/q}.\nonumber
		\end{align}
		Now we can apply Lemma~\ref{lemma:aux1} since  $\sigma>d\tau$. We get
		\beq \label{eq:est1} 
		\Big\| \sum_{j=0}^\infty \psi_j \ast g_j \mid \bt(\rn)\Big\| \lesssim  \sup_{P \in \mathcal{Q}} \frac{1}{|P|^\tau} \left\{\sum_{j=\max\{j_P, 0\}}^{\infty} \big\| 2^{js}\,\mathcal{M}^{g_j}_{N} \mid L_p(P) \big\|^{q} \right\}^{1/q} = \big\|\{ 2^{js}\, \mathcal{M}_{N}^{g_j}\}_{j \in \no} \mid \mixB \big\|.
		\eeq 
		
		\emph{Step 2.} Let $f \in \bt(\Omega)$.  
		Then, for any $\varepsilon \in (0, \infty)$, there exists an $h \in \bt(\rn)$ such that $h|_{\Omega}=f$ in ${\mathcal D}'(\Omega)$ and
		\beq
		\left\|h \mid \bt(\rn)\right\| \leq \left\|f \mid \bt(\Omega)\right\| + \varepsilon.
		\eeq
		We have $\phi_j \ast f(y)= \phi_j \ast h(y)$ if $y\in \Omega$  since $\supp \phi_j\subset -K$ and $y+K\subset \Omega$ for any point $y\in \Omega$. In consequence
		\begin{align}\label{phifphig}
			\sup_{y \in \Omega} \frac{|\phi_j \ast f(y)|}{(1+2^j|x-y|)^N}\,  
			\lesssim\displaystyle \sup_{y \in \rn} \frac{|\phi_j \ast h(y)|}{(1+2^j|\tilde{x}-y|)^N}, \quad & x \not\in \overline{\Omega},
		\end{align}
		where $\tilde{x}:=(x', 2\omega(x')-x_n) \in \Omega$ is the symmetric point to $x=(x',x_n) \not\in \overline{\Omega}$ with respect to $\partial \Omega$. \\

		Let $g_j := (\phi_j \ast f)_{\Omega}$ for all $j \in \no$.  It was proved in  \cite[pp.\,248]{Ryc} that
		\begin{align}\label{gjphif}
			\sup_{y \in \Omega} \frac{| g_j(y)|}{(1+2^j|x-y|)^N}\, 
			&\begin{cases}\displaystyle \,= \sup_{y \in \Omega} \frac{|\phi_j \ast f(y)|}{(1+2^j|x-y|)^N}, & x \in \Omega, \\ \, \lesssim \displaystyle\sup_{y \in \Omega} \frac{|\phi_j \ast f(y)|}{(1+2^j|\tilde{x}-y|)^N}, & x \not\in \overline{\Omega}.
			\end{cases}
		\end{align}
		Now, we conclude from  \eqref{eq:est1},  \eqref{phifphig} and \eqref{gjphif} that
		\begin{align*}
			\left\|\ext f\mid \bt(\rn)\right\| & = \Big\| \sum_{j=0}^{\infty} \psi_j \ast (\phi_j \ast f)_\Omega \mid \bt(\rn)\Big\| \lesssim 
			\left\| \{2^{js}\,\mathcal{M}_N^{g_j} 
			\}_{j \in \no} \mid \mixB\right\|\\  
			&\lesssim \sup_{P \in \mathcal{Q}} \frac{1}{|P|^\tau} \left\{ \sum_{j=\max\{j_P, 0\}}^{\infty} 2^{jsq} \Big\|\sup_{y \in \Omega} \frac{|\phi_j \ast h(y)|}{(1+2^j|\cdot -y|)^N} \mid L_p(P)\Big\|^q\right\}^{1/q}\\ 
			&\lesssim \sup_{P \in \mathcal{Q}} \frac{1}{|P|^\tau} \left\{ \sum_{j=\max\{j_P, 0\}}^{\infty} 2^{jsq} \Big\|\sup_{y \in \rn} \frac{|\phi_j \ast h(y)|}{(1+2^j|\cdot -y|)^N} \mid L_p(P)\Big\|^q\right\}^{1/q} .
		\end{align*}
		Thus the last inequalities, the  characterisation of $\bt(\rn)$ via the Peetre maximal functions stated in Theorem~\ref{Peetre}~(i) and the choice of $g$  imply that
		\beq\label{eq:final}
		\left\|\ext f\mid \bt(\rn)\right\| \lesssim \left\|h\mid \bt(\rn)\right\| \lesssim \left\|f\mid \bt(\Omega)\right\|+\varepsilon. 
		\eeq
		Letting $\varepsilon\rightarrow 0$, we then know that $\ext$ is a bounded linear operator from $\bt(\Omega)$ into $\bt(\rn)$. 
		
		Finally, since the supports of $\psi_0$ and $\psi$ lie in $-K$, it follows that
		\beq
		\ext f|_\Omega = \sum_{j=0}^{\infty} \psi_j \ast \phi_j \ast f =f\quad \mbox{in }{\mathcal D}'(\Omega).\nonumber
		\eeq
		Therefore, $\ext$ is the desired extension operator from $\bt(\Omega)$ into $\bt(\rn)$, which concludes the proof for the Besov-type spaces. \\
		
		\emph{Step 3.} The proof for the Triebel-Lizorkin-type spaces follows similarly. Therefore, we point out the differences without giving the details. The estimation of the norm $\big\|\psi_j \ast g_j \mid F^{s-2\sigma, \tau}_{p,q}(\rn)\big\|$ is the first difference, but it can be done similarly as in \eqref{eq:est0}. Afterwards, the counterpart of \eqref{eq:est1} can be obtained using again Lemma~\ref{lemma:aux1}, but now the estimate related to the spaces $\mixF$. Finally, the characterisation of $\ft(\rn)$ via the Peetre maximal functions stated in Theorem~\ref{Peetre}~(ii) leads us to obtain a similar estimate as \eqref{eq:final}.
	\end{proof}


\begin{remark}\label{remu}
	As mentioned  above  essentially the same proof for Triebel-Lizorkin-type spaces can be found in \cite{ZHS20} with additional restrictions $p,q\ge 1$.   
\ignore{	The extension operator $\ext$ given by \eqref{eq:ext} depends on $p$, $\tau$ and $s$. More precisely, for fixed $L_\phi$ and $L_\psi$ we know from  Theorem~\ref{Th:ext} that $\ext$  is the extension operator from $\bt(\Omega)$ into $\bt(\rn)$, if 
\[ L_\psi\ge L_\phi\ge \max\{\lfloor s\rfloor,  N +d\tau, +\lfloor -s\rfloor_+ \}, \qquad N \in \nn \cap \left( \frac{d}{\min\{1,p\}}, \infty\right). \] 	
\end{remark}

\begin{remark}}
Note that, in view of the coincidence \eqref{fte}, also the Triebel-Lizorkin-Morrey spaces $\MF$ are covered by our theorem. This complements the corresponding result obtained in \cite[Proposition~4.13]{YSY-15} for the class of Besov-Morrey spaces $\MB$.
Very recently similar arguments were used in \cite{ZHS20,Z21} for the construction of the extension operator, but restricted to the case $p,q\geq 1$. For the sake of completeness, we briefly sketched our proof here. 
\end{remark}
\begin{corollary}
		Let $\Omega \subset \rn$ be an interval if $d=1$ or a Lipschitz domain if $d\geq 2$. Then there exists a linear bounded operator $\ext$ which maps $\MF(\Omega)$ into $\MF(\rn)$ for all $s \in \real$, $q \in (0, \infty ]$ and $0<p\leq u< \infty$, such that, for all $f \in {\mathcal D}'(\Omega)$, $\ext f|_{\Omega}=f$ in ${\mathcal D}'(\Omega)$.
\end{corollary}

\begin{proof}
This follows immediately from Theorem~\ref{Th:ext} and \eqref{fte}.
  \end{proof}

\bigskip
{Using Theorem~\ref{Th:ext} and the wavelet decomposition of the spaces $\at(\rn)$, cf. Proposition~\ref{wav-type2}, we can now prove a result on the monotonicity of the spaces $\at(\Omega)$ regarding the parameter $\tau$, which in fact does not hold when considering the spaces on $\rn$. 
\begin{proposition}\label{taumonoton}
	Let $0< p\leq\infty$ ($p<\infty$ in the F-case), $s\in \real$, $0<q\leq\infty$,  $0\leq 
	\tau_2\le \tau_1$. Let $\Omega \subset \rn$ be {a bounded} interval if $d=1$ or {a bounded}  Lipschitz domain if $d\geq 2$. Then  
	\[
	\id_{\tau}: A^{s,\tau_1}_{p,q}(\Omega)\hookrightarrow A^{s,\tau_2}_{p,q}(\Omega).
	\]
\end{proposition}
\begin{proof}
	Let $\wz{Q}_0$ be a dyadic cube that contains $\overline{\Omega}$ in its interior and let  $\wz{Q}$ be a (fixed) dyadic cube that contains the supports of all the functions $\psi_{i,j,k}$ and $\phi_{0,k}$ with non-empty intersection with $\wz{Q}_0$.  Let $f\in  A^{s,\tau_1}_{p,q}(\Omega)$. The compactly supported smooth functions are pointwise multipliers in $ \at(\real^d)$, cf. \cite{s011} or \cite[Theorem~6.1]{ysy} for $\tau\le \frac{1}{p}$ and Proposition~\ref{yy02} for $\tau>\frac{1}{p}$, therefore 
	\begin{equation}\label{ls4_01}
	\|f \mid \at(\Omega )\| \sim \inf \{  \|g \mid \at(\real^d )\| :\quad g\in\at(\real^d )\quad\text{and}\quad\supp g\subset \wz{Q}_0\}. 
	\end{equation}
 There exists a dyadic cube  $\widetilde{Q}_1$ such that $\langle g,\psi_{i,j,m}\rangle = \langle g,\phi_{0,m}\rangle  =0$ for any $g\in A^{s,\tau_1}_{p,q}(\real^d)$ with $\supp g\subset \wz{Q}_0$ and  $\phi_{0,\ell}$, $\psi_{i,j,m}$ such that  $ \supp \phi_{0,\ell}\nsubseteq \widetilde{Q}_1$,  $ \supp \psi_{i,j,m}\nsubseteq \widetilde{Q}_1$. Moreover there exists a positive constant $C$ such that
    		\begin{equation*}
    		\frac{ 1}{|P|^{\tau_2}} \le  C\, \frac{ 1}{|P|^{\tau_1}}\quad \text{for }\quad 0\leq\tau_2 \le \tau_1\ , 
    	\end{equation*} 
    	 if $P\subset \wz{Q}_1$. Therefore
    \begin{align}\label{ls04_02}
    	\|\lambda(g)\mid & a^{s,\tau_2}_{p,q}\|^\ast \le c \|\lambda(g)\mid a^{s,\tau_1}_{p,q}\|^\ast .
    	    \end{align}

	By Theorem~\ref{Th:ext}  there exists a linear and bounded extension operator $\ext$ from $A^{s, \tau_1}_{p,q}(\Omega)$ into $A^{s, \tau_1}_{p,q}(\real^d)$. So, using also the wavelet decomposition of these spaces, if $\varphi\in C^\infty_0(\real^d)$ is supported in $\wz{Q}_0$ and equals $1$ on $\Omega$, then  by \eqref{ls04_02}
	\begin{align}\label{ls4_03}
	\|f \mid A^{s, \tau_2}_{p,q}(\Omega)\|  & \le  C\, \|\lambda(\varphi \ext(f))\mid a^{s, \tau_2}_{p,q}\|^\ast \le  C\, \|\lambda(\varphi \ext(f))\mid a^{s, \tau_1}_{p,q}\|^\ast   \\ 
	& \le C\,  \|\varphi \ext(f)\mid A^{s, \tau_1}_{p,q}(\real^d)\|  \le C \, \|f \mid A^{s, \tau_1}_{p,q}(\Omega)\|. \nonumber
	\end{align}
	The proof is complete.	
\end{proof}}


\section{Limiting embeddings}\label{lim-emb}


We shall always assume in the sequel that $\Omega$ is a bounded Lipschitz domain in $\rn$. 
As already  mentioned, we shall deal -- different from the standard approach -- with {\em continuous} embeddings of type 
\[
\id_\tau : \ate(\Omega) \hookrightarrow \atz(\Omega)
\]
only after we studied their {\em compactness} in \cite{ghs20}.

But it will turn out that only the limiting cases are of particular interest. So we collect first what is more or less obvious. Note that we use the above notation always with the understanding that either both, source and target space, are of Besov-type ($A=B$), or both are of Triebel-Lizorkin-type ($A=F$).

For convenience we use the following abbreviation:

\begin{align}
&\critical \nonumber\\
& :=  
\max\left\{\left(\tau_2-\frac{1}{p_2}\right)_+ -\left(\tau_1-\frac{1}{p_1}\right)_+, \frac{1}{p_1} -\tau_1 - \frac{1}{p_2}+\tau_2, 
\frac{1}{p_1}-\tau_1 - \min\left\{\frac{1}{p_2}-\tau_2, \frac{1}{p_2}(1-p_1\tau_1)_+\right\}\right\}\nonumber\\
\label{gamma}
& =  \begin{cases}
\frac{1}{p_1}-\tau_1-\frac{1}{p_2}+\tau_2, & \text{if}\quad \tau_2\geq \frac{1}{p_2}, \\[1ex]
\frac{1}{p_1}-\tau_1, &\text{if}\quad  \tau_1\geq \frac{1}{p_1}, \ \tau_2< \frac{1}{p_2}, \\[1ex]
\max\{0, \frac{1}{p_1}-\tau_1-\frac{1}{p_2}+\max\{\tau_2,\frac{p_1}{p_2}\tau_1\}\}, &\text{if}\quad  \tau_1< \frac{1}{p_1}, \  \tau_2< \frac{1}{p_2}.
\end{cases}
\end{align}
{Here and in the sequel we put $p_i\tau_i=1$ in case of $p_i=\infty$ and $\tau_i=0$. Similarly we shall understand  $\frac{p_i}{p_k}=1$ if $p_i=p_k=\infty$.}

\begin{theorem}[\cite{ghs20}]  \label{cont-tau}
	Let  $s_i\in \real$, $0<q_i\leq\infty$, $0<p_i\leq \infty$ (with $p_i<\infty$ in case of $A=F$), $\tau_i\geq 0$, $i=1,2$. 
	\bli
	\item[{\bfseries\upshape (i)}]
	The embedding 
	\[
	\id_{\tau}: \ate(\Omega )\hookrightarrow \atz(\Omega )
	\]
	is compact if, and only if, 
	\begin{equation*}
	\frac{s_1-s_2}{d} > \critical.
	\end{equation*}
	\item[{\bfseries\upshape (ii)}]
	There is no continuous embedding 
	\[
	\id_{\tau}: \ate(\Omega )\hookrightarrow \atz(\Omega )
	\]
	if
	\begin{equation}\label{no-cont-tau}
	\frac{s_1-s_2}{d} < \critical.
	\end{equation}
	\eli
\end{theorem}

This result was proved in \cite{ghs20} and shows us that we are indeed left to deal with the limiting case
\[
\frac{s_1-s_2}{d}  = \critical.
\]

First we prove the following lemma that extends \cite[Theorem~3.1]{hs12b}  to the case $u=p=\infty$. We recall that 
$\mathcal{N}^{s}_{\infty,\infty,q}(\rn)=  B^{s}_{\infty,q}(\rn) $. 

\begin{lemma}\label{lemmaNN}
Let  $s_i\in \real$, $0<q_i\leq\infty$, $0<p_i\leq u_i< \infty$  or $p_i=u_i=\infty$, $i=1,2$.
Then 
\begin{align}
	\MBe(\Omega) \hookrightarrow  \mathcal{N}^{s_2}_{\infty,\infty,q_2}(\Omega) & 
	\quad \mbox{if, and only if,} \quad
	 \frac{s_1-s_2}{d}> \frac{1}{u_1}\quad\text{or}\quad  \frac{s_1-s_2}{d}= \frac{1}{u_1}\quad\text{and}\quad q_1\le q_2 ,
	\label{N-Binfty-emb}
	\intertext{and}
     \mathcal{N}^{s_1}_{\infty,\infty,q_1}(\Omega)  \hookrightarrow	\MBz(\Omega) & \quad \mbox{if, and only if,} \quad
	s_1> s_2  \quad\text{or}\quad  s_1=s_2 \quad\text{and}\quad q_1\le q_2 .
	\label{Binfty-N-emb}
	\end{align}
\end{lemma}

\begin{proof}
The necessity of the conditions in \eqref{N-Binfty-emb} follows easily by the following chain of embeddings
\[ B^{s_1}_{u_1,q_1}(\Omega)\hookrightarrow \MBe(\Omega) \hookrightarrow  \mathcal{N}^{s_2}_{\infty,\infty,q_2}(\Omega)  = B^{s_2}_{\infty,q_2}(\Omega) \]
and the properties of embeddings of classical Besov spaces. Whereas the sufficiency can be proved in the same way as in the proof of \cite[Theorem~3.1]{hs12b}. 

To prove the  second embedding it is sufficient to note that 
 \[  \mathcal{N}^{s_1}_{\infty,\infty,q_1}(\Omega) = B^{s_1}_{\infty,q_1}(\Omega) \hookrightarrow	B^{s_2}_{u_2,q_2}(\Omega) \hookrightarrow\MBz(\Omega). \] 
 On the other hand it follows from \eqref{N-Binfty-emb} that 
 \[ B^{s_1+\frac{d}{u_1}}_{u_1,q_1}(\Omega)\hookrightarrow\mathcal{N}^{s_1}_{\infty,\infty,q_1}(\Omega) \hookrightarrow \MBz(\Omega).\]
So if the last embedding holds, then it follows from \cite[Corollary~3.7]{hs12b} that $s_1>s_2$, or $s_1=s_2$ and $q_1\le q_2$.
\end{proof}

\begin{proposition}  \label{lim-tau2-large}
	Let  $s_i\in \real$, $0<q_i\leq\infty$, $0<p_i\leq \infty$ (with $p_i<\infty$ in case of $A=F$), $\tau_i\geq 0$, $i=1,2$. Assume 
	\[
	\tau_2\geq \frac{1}{p_2}\quad \text{with}\quad q_2=\infty\quad \text{if}\quad \tau_2=\frac{1}{p_2}.
	\]
	Then the embedding 
	\[
	\id_{\tau}: \ate(\Omega )\hookrightarrow \atz(\Omega )
	\]
	is continuous if, and only if,
	\[ \frac{s_1-s_2}{d}  \geq  \critical.
	\]
\end{proposition}

\begin{proof} Note that by Theorem~\ref{cont-tau} we are left to deal with the limiting case $s_1-s_2  = d \critical$ only. 
  In view of {Proposition}~\ref{yy02} we always have $\atz(\Omega)=B^{s_2+d(\tau_2-\frac{1}{p_2})}_{\infty,\infty}(\Omega)$ now. Assume first 
	$\tau_1\geq \frac{1}{p_1}$ with $q_1=\infty$ if $\tau_1=\frac{1}{p_1}$, then by the same result also $\ate(\Omega)=B^{s_1+d(\tau_1-\frac{1}{p_1})}_{\infty,\infty}(\Omega)$ such that $\id_\tau$ is continuous if, and only if,
	\[
	\id:B^{s_1+d(\tau_1-\frac{1}{p_1})}_{\infty,\infty}(\Omega) \hookrightarrow 
	B^{s_2+d(\tau_2-\frac{1}{p_2})}_{\infty,\infty}(\Omega).
	\]
	But this is always true if $\frac{s_1-s_2}{d}=\critical=
	\frac{1}{p_1}-\tau_1-\frac{1}{p_2}+\tau_2$, recall Remark~\ref{R-lim-B}.\\
	
	Assume next $0\leq \tau_1<\frac{1}{p_1}$. We put $\frac{1}{u_1}=\frac{1}{p_1}-\tau_1$. We first show the sufficiency of $s_1-s_2= d\critical$ for the continuity of $\id_\tau$. We use  
	\eqref{N-BT-equal}, \eqref{elem-tau}, Proposition~\ref{yy02} and Lemma~\ref{lemmaNN} to obtain
	\[\ate(\Omega) 
	 \hookrightarrow  B^{s_1,\tau_1}_{p_1,\infty} (\Omega)=\mathcal{N}^{s_1}_{u_1,p_1,\infty}(\Omega)\hookrightarrow B^{s_2+d(\tau_2-\frac{1}{p_2})}_{\infty,\infty}(\Omega)=\atz(\Omega). 
	\]
	On the other hand, for the necessity, 
	\[\mathcal{N}^{s_1}_{u_1,p_1,\min\{p_1,q_1\}}(\Omega)\hookrightarrow\ate(\Omega)\hookrightarrow  \atz(\Omega)=B^{s_2+d(\tau_2-\frac{1}{p_2})}_{\infty,\infty}(\Omega)
	\]
	and \eqref{N-Binfty-emb} implies $s_1-s_2 \  = \ d\critical$ if $\id_\tau$ is  continuous.

	It remains to consider the case $\tau_1=\frac{1}{p_1}$, $q_1<\infty$. Now we benefit from the following chains of embeddings
\begin{align}\label{tlarge1}
\ate(\Omega )\hookrightarrow  B^{s_1+d(\tau_1-\frac{1}{p_1})}_{\infty,\infty}(\Omega)\hookrightarrow  B^{s_2+d(\tau_2-\frac{1}{p_2})}_{\infty,\infty}(\Omega) = \atz(\Omega )
\end{align}
{for the sufficiency of the condition $s_1-s_2= d\critical$, and }
\begin{equation}\label{lim-13}
B^{s_1}_{\infty,\min(p_1,q_1)}(\Omega) \hookrightarrow {\ate(\Omega) \hookrightarrow \atz(\Omega)} = B^{s_2+d(\tau_2-\frac{1}{p_2})}_{\infty,\infty}(\Omega),
\end{equation}
{for its necessity, cf. \eqref{elem-tau} and \cite[Proposition~2.4]{ysy}.} 
\end{proof}

\begin{remark}
	Note that the above result is the direct counterpart of our result for spaces on $\rn$ obtained in Theorem~\ref{B-rn}~(i), since $\critical =\frac{1}{p_1}-\tau_1-\frac{1}{p_2}+\tau_2 $ in the above setting.
\end{remark}

\begin{remark}\label{R-bmo-1}
    Recall the definition of the spaces $\bmo(\rn)$ in Remark~\ref{bmo-def} and define $\bmo(\Omega)$ by restriction, that is, in analogy to Definition~\ref{tau-spaces-Omega}. Then
    \begin{equation}\label{bmo=ft=bt-O}
      \bmo(\Omega)=F^{0,1/p}_{p,2}(\Omega) = {B^{0,1/2}_{2,2}(\Omega), \quad 0<p<\infty},
    \end{equation}
    extending \eqref{ft=bmo} and \eqref{ftbt} to domains $\Omega$.
  Taking $\bmo(\Omega)$ as the source space, that is, $\tau_1=\frac{1}{p_1}$, $s_1=0$, then Proposition~\ref{lim-tau2-large} implies that for $s\in\real$, 
$0<p,q\leq\infty$, and $\tau\geq \frac{1}{p}$ {with} $q=\infty$ if $\tau=\frac{1}{p}$, then 
  \[
\id_\tau:  \bmo(\Omega) \hookrightarrow \at(\Omega)
  \]
  is continuous if, and only if,  $s\leq -d (\tau-\frac1p)\leq 0$.  If $\bmo(\Omega)$ was the target space, then Proposition~\ref{lim-tau2-large} cannot be applied since $q_2=2<\infty$. 
\end{remark}

In view of Theorem~\ref{cont-tau} and Proposition~\ref{lim-tau2-large} we are left  to study the situation
\begin{equation}\label{ddh_4-1}
\frac{s_1-s_2}{d}  = \critical\quad\text{and}\quad \tau_2\leq \frac{1}{p_2}\quad \text{with}\quad q_2<\infty\quad \text{if}\quad \tau_2=\frac{1}{p_2}
\end{equation}
in the sequel.
Next we give some counterpart of Proposition~\ref{lim-tau2-large} dealing with the case when in the source space the parameter $\tau_1$ is large.


\begin{proposition}  \label{lim-tau1-large}
	Let  $s_i\in \real$, $0<q_i\leq\infty$, $0<p_i\leq \infty$, $\tau_i\geq 0$, $i=1,2$. Assume that \eqref{ddh_4-1} is satisfied and 
	\[
	\tau_1\geq \frac{1}{p_1}\quad \text{with}\quad q_1=\infty\quad \text{if}\quad \tau_1=\frac{1}{p_1}.
	\]
	Then the embedding 
	\[
	\id_{\tau}: \ate(\Omega )\hookrightarrow \atz(\Omega )
	\]
	is continuous if, and only if,
	\[
	 q_2=\infty .
	\]
\end{proposition}

\begin{proof}
  First note that $\critical = \frac{1}{p_1}-\tau_1$  by \eqref{gamma} and thus $\ate(\Omega)=B^{s_1+d(\tau_1-\frac{1}{p_1})}_{\infty,\infty}(\Omega)=B^{s_2}_{\infty,\infty}(\Omega)$ in view of Proposition~\ref{yy02}.

We first deal with the case $A=B$ and start with the sufficiency of $q_2=\infty$. Then Proposition~\ref{lim-tau2-large} covers the case $\tau_2 =\frac{1}{p_2}$ and we may assume $\tau_2<\frac{1}{p_2}$, recall \eqref{ddh_4-1}. But in view of \eqref{010319} and \eqref{N-BT-equal} we get 
	\[
	\bte(\Omega) \hookrightarrow B^{s_1+d(\tau_1-\frac{1}{p_1})}_{\infty,\infty}(\Omega) \hookrightarrow \mathcal{N}^{s_2}_{u_2,p_2,\infty}(\Omega)= B^{s_2,\tau_2}_{p_2,\infty}(\Omega). 
	\]

Now assume $A=F$ and again $\tau_2<\frac{1}{p_2}$.	If $f\in \fte(\Omega) =B^{s_1+d(\tau_1-\frac{1}{p_1})}_{\infty,\infty}(\Omega)$, recall Proposition~\ref{yy02}, then there exists some $g\in B^{s_1+d(\tau_1-\frac{1}{p_1})}_{\infty,\infty}(\rn)$ such that $f=g|_\Omega$  We can choose $g$ such that it can be represented as in \eqref{wavelet} with  the summation over $k\in \zn$ restricted to the indices $k$ such that $|k|\le K$ for some fixed $K$ since $\Omega $ is a bounded domain. Moreover we can choose $g$ in such a way that 
	\[ \|\lambda_k \mid \ell_\infty(\zn)\| + \sup_{j\in \nn} 2^{j(s_1+d(\tau_1-\frac{1}{p_1})+\frac{d}{2})} \sup_{i=1,\ldots ,2^d-1;\;k\in \zn} |\lambda_{i,j,k}| \le C \|f\mid B^{s_1+d(\tau_1-\frac{1}{p_1})}_{\infty,\infty}(\Omega)\| ,\]
	{for some constant $C>0$ independent of $f$.}  
        We have to show that $f\in F^{s_2,\tau_2}_{p_2,\infty}(\Omega)$  and $\|f \mid F^{s_2,\tau_2}_{p_2,\infty}(\Omega) \| \le c \|f\mid  B^{s_1+d(\tau_1-\frac{1}{p_1})}_{\infty,\infty}(\Omega)\| $. 
        It is sufficient to note that for any $i=1,\ldots, 2^d-1$,
\begin{align}\label{09.05}
\left\|\sup_{j,k}  2^{j(s_2-\frac{d}{u_2}+\frac{d}{2})} |\lambda_{i,j,k}|{2^{j\frac{d}{u_2}}\chi_{j,k}(\cdot)}\ignore{\chi^{(u_2)}_{j,k}(\cdot)}\mid {\mathcal{M}_{u_2,p_2}}{(\rn)}\right\| \le C \sup_{j\in \nn} 2^{j(s_1+d(\tau_1-\frac{1}{p_1})+\frac{d}{2})} \sup_{i=1,\ldots ,2^d-1;\;k\in \zn} |\lambda_{i,j,k}| .
\end{align}
The rest follows from the wavelet characterisation of $\mathcal{E}^{s_2}_{u_2,p_2,\infty}(\rn)={F^{s_2,\tau_2}_{p_2,\infty}}(\rn)$, $\frac{1}{u_2}=\frac{1}{p_2}-\tau_2$, cf. \cite{lsuyy}. But  \eqref{09.05} follows easily from the identities $s_2=s_1+d(\tau_1-\frac{1}{p_1})$ and $\|{2^{j\frac{d}{u_2}}\chi_{j,k}(\cdot)}\ignore{\chi^{(u_2)}_{j,k}(\cdot)}\mid \mathcal{M}_{u_2,p_2}{(\rn)}\|=1$.

Now we prove the necessity and assume that $\id_\tau: \ate(\Omega) \hookrightarrow \atz(\Omega)$ is continuous. We start with the case $A=B$. If $p_2=\infty$, then $\tau=0$ by assumption \eqref{ddh_4-1} and $B^{s_2,\tau_2}_{\infty,q_2}(\Omega)= B^{s_2}_{\infty,q_2}(\Omega)$. So both the source and the target space are classical Besov spaces and it is well-known that, in that case, $q_2=\infty$, recall Remark~\ref{R-lim-B}. So it remains to consider the case  $p_2<\infty$. To simplify the notation we assume that 
	the support of the wavelets $\psi_{i,j,m}$ such that  $Q_{j,m}\subset Q_{0,0}$ are contained in $\Omega$. If it is not true one can easily rescale the argument. 
	
	We take a sequence  $\lambda=(\lambda_{i,j,m}){_{i,j,m}, i=1,...,2^d-1, j\in \no,m\in \zn,}$ defined by the formula
	\[ \lambda_{i,j,m}=
	\begin{cases}
	2^{-j(s_2+\frac{d}{2})} & {\rm if }\quad i=1 \quad {\rm and}\quad Q_{j,m}\subset Q_{0,0},\\
	0 & {\rm otherwise}.
	\end{cases}
	\] 
Then, using the sequence space version of Proposition~\ref{yy02},
\begin{equation} \nonumber 
	\|\lambda\mid \sbte\| \sim \|\lambda\mid b^{s_1+d(\tau_1-\frac{1}{p_1})}_{\infty,\infty} \| = 1 
\end{equation}
since $s_2= s_1+d(\tau_1-\frac{1}{p_1})$. Here we used the notation $b^\sigma_{p,q} = b^{\sigma,0}_{p,q}$. On the other hand, for any dyadic cube $P\subset Q_{0,0}$ and any $j\ge j_P$, we have 
\begin{align*}
2^{j(s_2+\frac{d}{2}-\frac{d}{p_2})}\left(\sum_{Q_{j,m}\subset P}|\lambda_{i,j,m}|^{p_2}\right)^{\frac{1}{p_2}} = 2^{-j_P\frac{1}{p_2}}.  
\end{align*}
So if $q_2<\infty$, then 
\begin{equation}\nonumber
\|\lambda\mid b^{s_2,\tau_2}_{p_2,q_2} \| = \infty.
\end{equation}
Therefore, if $q_2<\infty$, the function
\[
f=\sum_{i,j,m} \lambda_{i,j,m} \psi_{i,j,m} 
\]
belongs to $B^{s_1,\tau_1}_{p_1,q_1}(\Omega)$ but not to $B^{s_2,\tau_2}_{p_2,q_2}(\Omega)$, which contradicts our assumption and thus finishes the proof of the necessity for the $B$-case.
	
The case $A=F$ follows by \eqref{elem-tau} and by what we just proved for the Besov-type spaces. Note that, for the $F$-spaces, we always have $p<\infty$. Therefore, by the following chain of embeddings 
\[
B^{s_1, \tau_1}_{p_1, \min\{p_1, q_1\}}(\Omega) \hookrightarrow \fte(\Omega) \hookrightarrow \ftz (\Omega) \hookrightarrow B^{s_2, \tau_2}_{p_2, \max\{p_2,q_2\}}(\Omega),
\] 
we obtain the necessity of the condition $\max\{p_2,q_2\}=\infty$, which here reads as $q_2=\infty$. 
\end{proof}

\begin{remark}
	Note that the above result differs from its $\rn$-counterpart in Theorem~\ref{B-rn}~(ii). In that case, there is never a continuous embedding in the setting of Proposition~\ref{lim-tau1-large}, that is, when conditions \eqref{ddh_4-1} and $\tau_1\geq\frac{1}{p_1}$ with $q_1=\infty$ when $\tau_1=\frac{1}{p_1}$ are satisfied.
\end{remark}

\begin{remark}\label{R-bmo-2}
    Again we return to the special case when the source or target space of $\id_\tau$ coincides with $\bmo(\Omega)$. Parallel to Remark~\ref{R-bmo-1} we cannot apply Proposition~\ref{lim-tau1-large} in case of $\ate(\Omega)=\bmo(\Omega)$. Otherwise, if $\atz(\Omega)=\bmo(\Omega)$, then 
    Proposition~\ref{lim-tau1-large} implies that there is never a continuous embedding of type
    \[\id_\tau : \at(\Omega) \hookrightarrow \bmo(\Omega)\]
    in the limiting case \eqref{ddh_4-1} which reads here as $0<p\leq \infty$ (with $p<\infty$ in case of $A=F$), $0<q\leq \infty$, $\tau\geq 0$ and $s = d(\frac1p-\tau)$. Moreover, there is no such continuous embedding whenever $\tau_2=\frac{1}{p_2}$ and $q_2<\infty$ in the limiting case \eqref{ddh_4-1}.
\end{remark}

For the rest of this section we shall now assume that
\begin{equation}\label{ddh_4-3}
0\leq  \tau_i\leq \frac{1}{p_i}\quad \text{with}\quad q_i<\infty\quad \text{if}\quad \tau_i=\frac{1}{p_i}, \quad i=1,2, \quad{\rm and}\quad \tau_1+\tau_2>0 , 
\end{equation}
 and thus
\begin{equation}
\frac{s_1-s_2}{d}  = \critical = \max\left\{0, \frac{1}{p_1}-\tau_1-\frac{1}{p_2}+\max\left\{\tau_2,\frac{p_1}{p_2}\tau_1\right\}\right\}.
\label{ddh_4-4}
\end{equation}

In view of the embeddings and coincidences \eqref{N-BT-emb}, \eqref{N-BT-equal} and \eqref{fte}, together with our previous findings for the spaces $\MA(\Omega)$ in \cite{hs12b,hs14} (as well as some $\rn$-counterparts of $\at(\rn)$ in Theorems~\ref{B-rn} and \ref{F-rn}), we expect some $q$-dependence now. For the moment, we restrict ourselves to the case of $\bt$ spaces.

\begin{theorem}  \label{lim-lim}
	Let $0< p_1,p_2\leq\infty$, $s_i\in \real$, $0<q_i\leq\infty$,  $0\leq 
	\tau_i\leq \frac{1}{p_i}$, 
	$i=1,2$. 
	Assume that the conditions \eqref{ddh_4-3} hold and that 
	\[
	\frac{s_1-s_2}{d}  = \critical =\max\left\{0, \frac{1}{p_1}-\tau_1-\frac{1}{p_2}+\max\left\{\tau_2,\frac{p_1}{p_2}\tau_1\right\}\right\}.
	\]
	\begin{itemize} 
	\item[{\bfseries\upshape (i)}] The embedding 
	\begin{equation}\label{2.07.20}
	\id_{\tau}: \bte(\Omega )\hookrightarrow \btz(\Omega )
	\end{equation}
	is continuous if one of the following conditions holds:
	\begin{align}
	&\frac{s_1-s_2}{d} = \frac{1}{p_1}-\tau_1-\frac{1}{p_2}+\tau_2>0, 
	\quad\text{and}\quad 
	 p_1\tau_1<p_2\tau_2,
	 \label{520-1}\\
      \mbox{or} \qquad & \frac{s_1-s_2}{d}=\frac{1}{p_1}-\tau_1-\frac{1}{p_2}+\frac{p_1}{p_2}\tau_1>0 \quad \text{and}\quad  q_1\le \frac{p_1}{p_2}q_2 , \label{520-4}\\
       \mbox{or} \qquad & s_1=s_2 \quad\text{and}\quad q_1\le \min\left\{1, \frac{p_1}{p_2}\right\}q_2.\label{520-5}	
	\end{align}
	\item[{\bfseries\upshape (ii)}] 
	 If the embedding \eqref{2.07.20} is continuous and one of the following conditions holds
	  \begin{align}
	    & \critical = 0, \label{ddh_4-5}\\ 
	   \mbox{or}\qquad & \critical= \frac{1}{p_1}-\tau_1-\frac{1}{p_2}+\frac{p_1}{p_2}\tau_1>0 , \label{ddh_4-6}
	  \end{align} 
	  then $q_1\le q_2$. 
	 
	 If the embedding \eqref{2.07.20} is continuous for any $q_1$ and $q_2$  with fixed $s_1,s_2,p_1,p_2,\tau_1,\tau_2$,  then \eqref{520-1} holds.
\end{itemize}
\end{theorem}

\begin{remark}\label{lim-rem}
	As mentioned above, we are left to consider the embedding in the limiting case \eqref{ddh_4-4} when \eqref{ddh_4-3} is satisfied. However, in case of $\tau_i=\frac{1}{p_i}$, $q_i=\infty$, for $i=1$ or $i=2$, the above Theorem~\ref{lim-lim} coincides with Propositions~\ref{lim-tau2-large} or \ref{lim-tau1-large}, respectively. So in fact situation \eqref{ddh_4-3} is the only interesting one now. 
	
	We have always $p_1<p_2$ in \eqref{520-4} so we have a small  gap between sufficient and necessary conditions on $q_i$ here. We meet a   similar situation if $s_1=s_2$, $p_1<p_2$,  $\tau_1=\frac{1}{p_1}$  and $q_1<\infty$ in \eqref{520-5}. In all other cases the result is sharp. 
\end{remark}

\begin{proof}
	\emph{Step 1.} {We start by proving part (i). For this, we use an argument similar to the one used in the proof of Proposition~\ref{taumonoton}, based on the extension operator from Theorem~\ref{Th:ext} and the wavelet decomposition of the spaces $\bt(\rn)$, cf. Proposition~\ref{wav-type2}.} We use the same notation as there.  Let us denote by $\wz{\sbt}(\wz{Q}_0)$ the sequence space defined by
	\begin{equation}
	  \wz{\sbt}(\wz{Q}_0):= \left\{t=\{t_{i,j,m}\}_{i,j,m}:\, t_{i,j,m} \in \cc,\, j \in \no, i=1, \dots, 2^d-1, \ m\in\zn, Q_{j,m} \subset \wz{Q}_0,\, \|t \mid \wz{\sbt}\| <\infty\right\},\nonumber
	\end{equation}
	where
	\begin{equation}\label{ls20.09_1}
	\|t \mid\wz{\sbt}\| := \sup_{P\in\mathcal{Q};\; P\subset \wz{Q}_0} \frac{ 1}{|P|^{\tau}}\left\{\sum_{j=\max\{j_P,0\}}^\fz
	2^{j(s+\frac d2-\frac{d}{p})q} \sum_{i=1}^{2^d-1}
	\left[\sum_{m:\ Q_{j,m}\subset P}
	|t_{i,j,m}|^{p}\right]^{\frac{ q}{p}}\right\}^{\frac{1}{q}}. 
	\end{equation}
	Then, we have to prove that for some $C>0$ 
	\begin{equation} \label{ls04_04}
	\| t \mid \wz{\sbtz}\| \leq C \, 	\| t \mid \wz{\sbte}\|
	\end{equation}
	holds true for all $t\in \wz{\sbte}$.	Here and in the sequel we assume for convenience that $p_i, q_i<\infty$, otherwise the modifications are obvious. 
	Please note, once more, that the assumption $P\subset \wz{Q}_{0}$  implies that  
	\begin{equation}\label{ls04_05}
	\frac{ 1}{|P|^{a}} \le  C \frac{ 1}{|P|^{b}}\quad \text{if }\quad a\le b   \qquad \text{and}\quad  \#\{m:\;Q_{j,m}\subset P \} \sim 2^{jd}\min\{1, 2^{-j_P d}\}\quad \text{if} \quad j\ge j_P. 
	\end{equation}
	Moreover, if $1\le |P|\le | \wz{Q}_0|$, then 
	\begin{equation}\label{ls04_06}
	\frac{ 1}{|P|^{\tau_2}} \sim \frac{ 1}{|P|^{\tau_1}}
	\end{equation}
	for any $\tau_1$ and $\tau_2$.  To shorten the notation we put $\gamma= \critical$.\\
	
  \emph{Substep 1.1.}  If  $\gamma=\frac{1}{p_1}-\tau_1-\frac{1}{p_2}+\tau_2>0$, then $\frac{p_1}{p_2}\tau_1\le \tau_2$. In this case   the statement follows from Theorem~\ref{Th:ext} and Theorem~\ref{B-rn}~(iii) if $\tau_1\not= \tau_2$.
	
	Indeed  Theorem~\ref{Th:ext} and Remark~\ref{remu} imply that there exists a common bounded  extension operator $\ext$ for the spaces $\bte(\Omega)$ and $\btz(\Omega)$ and we thus have the following commutative diagram
	\[
	\begin{CD}
	\bte(\Omega)@>{\rm id}>> \btz(\Omega)\\
	@V\ext VV  @AA{\rm re}A\\
	\bte(\rn)@>{\rm id}>> \btz(\rn).
	\end{CD} 
	\]
	Now the case \eqref{520-1} 
	follows from Theorem~\ref{B-rn}~(iii), as well as the situation when
	\begin{equation} \label{520-2-1}
	\frac{s_1-s_2}{d}=\frac{1}{p_1}-\tau_1-\frac{1}{p_2}+\tau_2>0,\quad \tau_1\not= \tau_2,\quad 
	\frac{p_1}{p_2}=\frac{\tau_2}{\tau_1}\quad\text{and}\quad 
	q_1\le\frac{p_1}{p_2}q_2. 
	\end{equation}

	{ \emph{Substep 1.2.} Let $\gamma=\frac{1}{p_1}-\tau_1-\frac{1}{p_2}+\tau_2>0$ and $\tau_1=\tau_2$. In that case $p_1<p_2$ and $s_1-\frac{d}{p_1}=s_2-\frac{d}{p_2}$.
	Let $t = \{t_{i,j,m}\}_{i,j,m}\in \wz{\sbte}$	 and let $\|t\mid\wz{\sbte}\|=1$. To simplify the notation we put $\lambda_{i,j,m}= 2^{-j(s_1+\frac{d}{2}-\frac{d}{p_1})}t_{i,j,m}$ and $\tau=\tau_1=\tau_2$.  
	
	For any $i$, $j$ and $m$ we have 
	\[ 
	|\lambda_{i,j,m}|\le 2^{-jd\tau} ,
	\]
and in consequence		
	\begin{equation}
	|\lambda_{i,j,m}|^{p_2}\le |\lambda_{i,j,m}|^{p_1} 2^{-jd\tau(p_2-p_1)}.
	\end{equation}
      In a parallel way, for any dyadic  cube $P\subset \wz{Q}_0$ we have  	
	\begin{equation}\label{ls20.09-2} 
	\sum_{Q_{j,m}\subset P} |\lambda_{i,j,m}|^{p_1} \le 2^{-j_Pd\tau p_1},
\end{equation}
so in consequence
\begin{align}
 \left( \sum_{Q_{j,m}\subset P} |\lambda_{i,j,m}|^{p_2}\right)^{\frac{q_2}{p_2}} \le
	\left( \sum_{Q_{j,m}\subset P} |\lambda_{i,j,m}|^{p_1}\right)^{\frac{q_2}{p_2}} 2^{-jd\tau(p_2-p_1)\frac{q_2}{p_2}} \le   \ c\ 
		 2^{-j_Pd\tau \frac{p_1}{p_2}q_2} 2^{-jd\tau(1-\frac{p_1}{p_2})q_2}
\end{align} 
for any $i=1,\dots, 2^d-1$. Summing up over $j$ we get
\begin{align}
\sum_{j=j_P}^\infty \sum_{i=1}^{2^d-1} \left( \sum_{Q_{j,m}\subset P} |\lambda_{i,j,m}|^{p_2}\right)^{\frac{q_2}{p_2}} 
\le   c\ 2^d 2^{-j_Pd\tau \frac{p_1}{p_2}q_2}    \sum_{j=j_P}^\infty 2^{-jd\tau(1-\frac{p_1}{p_2})q_2} = C  2^{-j_Pd\tau q_2}.
\end{align}
This proves that $t = \{t_{i,j,m}\}_{i,j,m}\in \wz{\sbtz}$ and $\|t \mid \wz{\sbtz}\|\le C$. 
}		
	
\emph{Substep 1.3.} Let  $\gamma=0\;$  {i.e., $s_1=s_2$}. Then $p_2\le p_1$ and $ \frac{1}{p_1}-\frac{1}{p_2} \le  \tau_1-\tau_2$ or $p_2 > p_1$ and $\tau_1=\frac{1}{p_1}$. First  we assume that $p_2\le p_1$ and $ \frac{1}{p_1}-\frac{1}{p_2} \le  \tau_1-\tau_2$. 	We conclude by H\"older's inequality for any $i=1, \dots, 2^d-1$, that
	\begin{align}
	\left[\sum_{m:\ Q_{j,m}\subset P} |t_{i,j,m}|^{p_2}\right]^{\frac{1}{p_2}} \le 
	2^{d(j-j_P)(\frac{1}{p_2}-\frac{1}{p_1})}
	\left[\sum_{m:\ Q_{j,m}\subset P} |t_{i,j,m}|^{p_1}\right]^{\frac{1}{p_1}} \, .
	\end{align}
	
	In consequence, for any $q\in (0,\infty]$,  
	\begin{align}\label{ls04_07}
	\left\{\sum_{j=\max\{j_P,0\}}^\fz
	2^{j(s_2+\frac d2-\frac{d}{p_2})q} \sum_{i=1}^{2^d-1} \left[\sum_{m:\ Q_{j,m}\subset P} |t_{i,j,m}|^{{p_2}}\right]^{\frac{q}{p_2}}
	\right\}^{\frac{1}{q}} \,\le  2^{j_P d(\frac{1}{p_1}-\frac{1}{p_2})}  
	\\ \times 
	\left\{\sum_{j=\max\{j_P,0\}}^\fz
	2^{j(s_1+\frac d2-\frac{d}{p_1})q}  2^{j(s_2-s_1)q}
	\sum_{i=1}^{2^d-1} \left[\sum_{m:\ Q_{j,m}\subset P} |t_{i,j,m}|^{p_1}\right]^{\frac{q}{p_1}} \right\}^{\frac{1}{q}}\, .\nonumber
	\end{align}
	If $\critical=0$, then $s_1=s_2$ and $a=\frac{1}{p_1}-\frac{1}{p_2}+\tau_2\le \tau_1=b$. So \eqref{ls04_04} follows from \eqref{ls04_07} and \eqref{ls04_05}-\eqref{ls04_06} for any $q_1\leq q_2$. 
	
	Now let $\gamma=0$, $p_2 > p_1$ and $\tau_1=\frac{1}{p_1}$.  First we consider the case $\tau_2=\frac{1}{p_2}$.  Let  
	\begin{equation} \label{assump1}
		\| t \mid {\wz{\sbte}}\|=1,
	\end{equation}
	which implies that, for every cube $P \in \mathcal{Q}, P \subset \wz{Q}_{0}$ and for every $j \geq \max\{j_P,0\}$, we have
	$$
	\frac{ 1}{|P|^{\tau_1 q_1}} \, 2^{j(s_1+\frac d2-\frac{d}{p_1})q_1} \sum_{i=1}^{2^d-1} \left[\sum_{m:\ Q_{j,m}\subset P} |t_{i,j,m}|^{p_1}\right]^{\frac{q_1}{p_1}} \leq 1
	$$
	In particular, we know that for every cube $Q_{\nu,m} \in \mathcal{Q}, Q_{\nu,m} \subset \wz{Q}_{0}$, with $\nu\geq 0$ and for every $i=1, ..., 2^d-1$, we have
	\begin{equation*}
		2^{\nu(s_1 + \frac{d}{2}- \frac{d}{p_1}+d \tau_1)}|t_{i,\nu,m}|=  2^{\nu(s_1 + \frac{d}{2})}|t_{i,\nu,m}| \leq 1.
	\end{equation*}
	So the condition $p_1< p_2$ implies
	\begin{equation}\label{assump2}
		2^{\nu( s_1 + \frac{d}{2})p_2}|t_{i,\nu,m}|^{p_2} \leq 2^{\nu( s_1 + \frac{d}{2})p_1}|t_{i,\nu,m}|^{p_1}.
	\end{equation}
	 We have to prove that $\| t \mid {\wz{\sbtz}}\|\lesssim 1$. 
	Let us fix a cube $P\in \mathcal{Q}, P \subset \wz{Q}_{0}$. Thus, by the inequality \eqref{assump2}, it follows that
	for $q_1=\frac{p_1}{p_2}q_2$ we have 
	\begin{align*}
	&\sum_{j=\max\{j_P,0\}}^\fz 2^{j(s_2+\frac d2-\frac{d}{p_2})q_2} \sum_{i=1}^{2^d-1} \left[\sum_{m:\ Q_{j,m}\subset P} |t_{i,j,m}|^{p_2}\right]^{\frac{ q_2}{p_2}}\\
	&= 
	\sum_{j=\max\{j_P,0\}}^\fz 2^{-j \frac{d}{p_2}q_2} \sum_{i=1}^{2^d-1} \left[\sum_{m:\ Q_{j,m}\subset P}2^{j(s_1+\frac d2)p_2} |t_{i,j,m}|^{p_2}\right]^{\frac{ q_2}{p_2}}\\
	& \leq 
	\sum_{j=\max\{j_P,0\}}^\fz 2^{-j\frac{d}{p_1} q_1} \sum_{i=1}^{2^d-1} \left[\sum_{m:\ Q_{j,m}\subset P}2^{j(s_1+\frac d2)p_1} |t_{i,j,m}|^{p_1}\right]^{\frac{ q_1}{p_1}}\\
	&=  
	\sum_{j=\max\{j_P,0\}}^\fz 2^{j(s_1+\frac d2-\frac{d}{p_1})q_1} \sum_{i=1}^{2^d-1} \left[\sum_{m:\ Q_{j,m}\subset P} |t_{i,j,m}|^{p_1}\right]^{\frac{ q_1}{p_1}}\\
	&\leq  	\| t \mid {\wz{\sbte}}\|^{q_1}\, |P|^{\tau_1 q_1}.   
	\end{align*}	
In consequence for any cube  $P\subset \wz{Q}_{0}$ we have 
\begin{align*}
\frac{ 1}{|P|^{\tau_2}} 
	& \left(\sum_{j=\max\{j_P,0\}}^\fz 2^{j(s_2+\frac d2-\frac{d}{p_2})q_2} \sum_{i=1}^{2^d-1} \left[\sum_{m:\ Q_{j,m}\subset P} |t_{i,j,m}|^{p_2}\right]^{\frac{ q_2}{p_2}}\right)^\frac{1}{q_2}\le  |P|^{\frac{\tau_1  q_1}{q_2}-\tau_2 } \le  1\\
\end{align*}
since  $\tau_1  q_1  -\tau_2q_2 = 0$ . 
 	So by monotonicity if $\tau_2=\frac{1}{p_2}$,  then for any $q_1$ and $q_2$ such that $q_1 \le  \frac{p_1}{p_2}q_2$ ,
\begin{equation*}
	\Big\| t \mid  {\wz{b^{s_2,\frac{1}{p_2}}_{p_2,q_2}}}\Big\| \leq  \| t \mid {\wz{\sbte}}\|.
	\end{equation*}	
If $\gamma=0$ and $\tau_1=\frac{1}{p_1}$, then $\tau_2\le \frac{1}{p_2}$. If $\tau_2<\frac{1}{p_2}$, then it follows from Substep 1.2 that 
	\[
	 \| t \mid {\wz{\sbtz}}\| \le C\, \Big\| t \mid {\wz{ b^{s_2,\frac{1}{p_2}}_{p_2,q_2}}}\Big\|,
	\]
so the final statement follows from the last two inequalities. 	

	
	{\em Substep 1.4.}~ Now let  $\gamma= \frac{1}{p_1}-\tau_1-\frac{1}{p_2}+\tau_1\frac{p_1}{p_2} > 0$ and $0\le \tau_1\le \frac{1}{p_1}$.  Please  note that this assumption implies $\tau_1\frac{p_1}{p_2} \ge  \tau_2$, $p_1\tau_1<1$  and $p_1<p_2$. 
	
The case  $\tau_1\frac{p_1}{p_2}=  \tau_2$ is covered by  \eqref{520-2-1} since $\tau_2=\tau_1\frac{p_1}{p_2} < \tau_1$. 
Let $\gamma>0$ and  $\tau_2 < \frac{p_1}{p_2}\tau_1$. We take $\tau_0$ such that  $\tau_0=\frac{p_1}{p_2}\tau_1$. The above considerations show that 
\[
\| t \mid {\wz{b^{s_2,\tau_0}_{p_2,q_2}}}\| \leq  C \, \| t \mid {\wz{\sbte}}\|
\]
if $q_1\le \frac{p_1}{p_2}q_2$. Since $\tau_2< \tau_0$ it follows from Substep 1.2 that 
\[
\| t \mid {\wz{\sbtz}}\| \le C \,\| t \mid {\wz{b^{s_2,\tau_0}_{p_2,q_2}}}\|.  
\]



    {\em Step 2.}~ Now we come to the necessity. We do some preparation first.\\
    
    By the diffeomorphic properties of Besov-type spaces, using translations and dilations if necessary we can assume that the domain $\Omega$ satisfies the following conditions: there exists some number $\nu_0\in\zz$ such that
    \begin{itemize}
    	\item $Q_{\nu_0,0} \subset \Omega$,
    	\item if $Q_{j,m} \subset Q_{\nu_0,0}, \quad j\geq 0, \quad$ then $\quad \supp \psi_{i,j,m} \subset \Omega$,
    	\item if $Q_{0,m} \subset Q_{\nu_0,0}, \quad$ then $\quad \supp \phi_{0,m} \subset \Omega$.
    \end{itemize}
    Due to the isomorphism resulting from the wavelet decomposition between function and sequence spaces, and similar to the explanation given in Substep~2.1 of the proof of \cite[Theorem~3.1]{hs12b}, one can equivalently prove the necessary conditions for the embedding
    $$
    \wz{\sbte}(Q_{\nu_0,0}) \hookrightarrow \wz{\sbtz}(Q_{\nu_0,0}),
    $$
    with $\nu_0<0$. For convenience, let us denote $\wz{Q}=Q_{\nu_0,0}$. \\
    
    {\em Substep 2.1.}~ We show that $q_1\leq q_2$ is necessary when $s_1=s_2$. We assume $q_1 >q_2$. Then we can choose a sequence of positive numbers $\{\gamma_j\}_{j\in\no} \in \ell_{q_1}(\no)\setminus\ell_{q_2}(\no)$. Let us define the sequence $t=\{t_{i,j,m}\}_{i,j,m}$, {$i=1,...,2^d-1, j\in \no, m \in \zn,$} by
    \begin{eqnarray}
    t_{i,j,m} := \begin{cases}
    2^{-j(s_1 + \frac d2)} \gamma_j \quad  &\mbox{if}\quad i=1 \quad \mbox{and} \quad Q_{j,m}\subset[0,1)^d,\\
    0 &\mbox{otherwise}.
    \end{cases}
    \end{eqnarray}
    Then, 
    \begin{align*}
    \|t \mid \wz{\sbte}\| &= \sup_{P\in\mathcal{Q};\; P\subset \wz{Q}}\frac{ 1}{|P|^{\tau_1}} 
    \left(\sum_{j=\max\{j_P,0\}}^\fz 2^{j(s_1+\frac d2-\frac{d}{p_1})q_1} \sum_{i=1}^{2^d-1} \left[\sum_{m:\ Q_{j,m}\subset P} |t_{i,j,m}|^{p_1}\right]^{\frac{ q_1}{p_1}}\right)^\frac{1}{q_1}\\
    &=  \sup_{P\in\mathcal{Q};\; P\subset \wz{Q}}\frac{ 1}{|P|^{\tau_1}} 
    \left(\sum_{j=\max\{j_P,0\}}^\fz 2^{j(s_1+\frac d2)q_1} \sum_{i=1}^{2^d-1} \left[ \int_P \left(\sum_{m \in \zn} |t_{i,j,m}|\chi_{j,m}(x)\right)^{p_1}\dint x \right]^{\frac{ q_1}{p_1}}\right)^\frac{1}{q_1}\\
    &= \sup_{P\in\mathcal{Q};\; P\subset \wz{Q}}\frac{ 1}{|P|^{\tau_1}} 
    \left(\sum_{j=\max\{j_P,0\}}^\fz 2^{j(s_1+\frac d2)q_1} 2^{-j(s_1 + \frac d2)q_1} |\gamma_j|^{q_1} |P\cap[0,1)^{d}|^{\frac{q_1}{p_1}}\right)^\frac{1}{q_1}\\
    &= \| \{\gamma_j\}_{{j\in \no}} \mid \ell_{q_1}\| < \infty,
    \end{align*}
   where the last equality holds because $\tau_1 \leq \frac{1}{p_1}$. On the other side, we obtain similarly that
   \begin{align*}
   \|t \mid \wz{\sbtz}\| &= \sup_{P\in\mathcal{Q};\; P\subset \wz{Q}}\frac{ 1}{|P|^{\tau_2}} 
   \left(\sum_{j=\max\{j_P,0\}}^\fz 2^{j(s_2+\frac d2-\frac{d}{p_2})q_2} \sum_{i=1}^{2^d-1} \left[\sum_{m:\ Q_{j,m}\subset P} |t_{i,j,m}|^{p_2}\right]^{\frac{ q_2}{p_2}}\right)^\frac{1}{q_2}\\  
   &=\| \{\gamma_j\}_{{j\in \no}} \mid \ell_{q_2}\|  = \infty, 
  \end{align*}
   which contradicts the embedding. \\
   
   {\em Substep~2.2.} Now we show that the condition $q_1\leq q_2$ is also necessary when $\frac{s_1-s_2}{d}= \frac{1}{p_1}-\tau_1 - \frac{1}{p_2}+ \frac{p_1}{p_2}\tau_1>0$. Let us assume $q_1>q_2$. We adapt the counter-example used in Substep~2.4 of the proof \cite[Theorem~3.2]{hs12}. For any $0>\nu\geq \nu_0$, we put 
   $$
   k_\nu := \lfloor2^{d|\nu|p_1 \tau_1}\rfloor,
   $$
   where $\lfloor x \rfloor=\max\{l \in \zz: l\leq x\}$. Then $1\leq k_\nu <2^{d|\nu|}$ and
   \begin{equation} \label{k-est}
   		k_\nu \leq c_{p_1, \tau_1} \, 2^{d(\mu - \nu)}\,k_{\mu}, \quad \mbox{if}\quad \nu\leq \mu<0. 
   \end{equation}
   For convenience, we assume $c_{p_1, \tau_1}=1$ (otherwise the proper modifications have to be done). As there, we define a sequence $t^{(\nu)}= \{t_{i,j,m}^{(\nu)}\}_{i,j,m}$, {$i=1,...,2^d-1, j\in \no, m \in \zn,$} in the following way: we assume that $k_\nu$ elements of the sequence equal $1$ and the rest equals $0$. If $j\neq 0$, $i\neq 1$ or $Q_{0,m} \nsubseteq Q_{\nu,0}$, then $t_{i,j,m}^{(\nu)}=0$. Because of \eqref{k-est}, we can choose the elements that equal 1 in such a way that the following property holds:
   \begin{equation*}
   \mbox{if} \quad Q_{\mu, l}\subseteq Q_{\nu,0} \quad \mbox{and} \quad Q_{\mu,l}= \bigcup_{i=1}^{2^{-d \mu}} Q_{0,m_i}, \quad \mbox{then at most } k_\mu \mbox{ elements } t_{1,0,m_i}^{(\nu)}  \mbox{ equal } 1.
   \end{equation*}
   Now we define a new sequence $t=\{t_{i,j,m}\}_{i,j,m} \in \wz{\sbte}$ by
   \begin{equation*}
   t_{i,j,m}= \gamma_j \, t_{i,0,m}^{(\nu)}, \quad \mbox{if}\quad  j=\nu-\nu_0 \quad \mbox{and} \quad Q_{0,m}\subset Q_{\nu,0}, 
   \end{equation*}
	where $\{\gamma_j\}_{j \in \no}$ is a sequence of positive numbers with $\{2^{j(s_1+\frac{d}{2}-\frac{d}{p_1}+\tau_1)}\, \gamma_j\}_{j \in \no} \in \ell_{q_1}(\no)\setminus\ell_{q_2}(\no)$. If $Q_{\mu,l}\subset Q_{\nu_0, 0}$, then for fixed $j \geq \mu$, there are at most $k_{\mu-l}$ non-zero elements $t_{i,j,m}$ such that $Q_{j,m} \subset Q_{\mu,l}$. Thus
	\begin{equation*}
		\sum_{m: Q_{j,m} \subset Q_{\mu,l}} |t_{i,j,m}|^{p_1} \leq \gamma_j^{p_1}\, 2^{d(j-\mu)\tau_1 p_1}
	\end{equation*}
	and the last sum is $k_{\mu-j}\gamma_j^{p_1}$ if $\mu=\nu_0$. Therefore, 
	\begin{align}\label{est1}
	\|t \mid \wz{\sbte}\| &= \sup_{P\in\mathcal{Q};\; P\subset \wz{Q}} \frac{ 1}{|P|^{\tau_1}} 
	\left(\sum_{j=\max\{j_P,0\}}^\fz 2^{j(s_1+\frac d2-\frac{d}{p_1})q_1} \sum_{i=1}^{2^d-1} \left[\sum_{m:\ Q_{j,m}\subset P} |t_{i,j,m}|^{p_1}\right]^{\frac{ q_1}{p_1}}\right)^\frac{1}{q_1}\nonumber\\
	& \leq \sup_{P\in\mathcal{Q};\; P\subset \wz{Q}} \frac{ 1}{|P|^{\tau_1}} 
	\left(\sum_{j=\max\{j_P,0\}}^\fz 2^{j(s_1+\frac d2-\frac{d}{p_1})q_1}\, \gamma_j^{q_1}\, 2^{d(j-j_P)\tau_1q_1} \right)^\frac{1}{q_1}\nonumber\\
	&= \sup_{P\in\mathcal{Q};\; P\subset \wz{Q}} 	\left(\sum_{j=\max\{j_P,0\}}^\fz 2^{j(s_1+\frac d2-\frac{d}{p_1}+d\tau_1)q_1} \,\gamma_j^{q_1} \right)^\frac{1}{q_1}\nonumber\\
	&= \|\{2^{j(s_1+\frac d2-\frac{d}{p_1}+d\tau_1)}\, \gamma_j\}_{{j\in\no}} \mid \ell_{q_1}\| <\infty.
	\end{align}
	Similarly we obtain
	\begin{equation*}
	\left(\sum_{m:\ Q_{j,m}\subset Q_{\mu,l}} |t_{i,j,m}|^{p_2} \right)^{\frac{q_2}{p_2}} \leq\left(\gamma_j^{p_2}\, k_{\mu-j}\right)^{\frac{q_2}{p_2}}\leq \gamma_j^{q_2}\, 2^{d(j-\mu)\tau_1 p_1\frac{q_2}{p_2}},
	\end{equation*}
	and when $\mu=\nu_0$
	\begin{equation*}
	\left(\sum_{m:\ Q_{j,m}\subset Q_{\nu_0,0}} |t_{i,j,m}|^{p_2} \right)^{\frac{q_2}{p_2}} =\left(\gamma_j^{p_2}\, k_{\nu_0-j}\right)^{\frac{q_2}{p_2}}\geq C\, \gamma_j^{q_2} \,2^{d(j-\nu_0)\tau_1 p_1\frac{q_2}{p_2}},
	\end{equation*}
	for some constant $C$ independent of $\gamma$. Then, 
	\begin{equation*}
	C\, \gamma_j^{q_2}\, 2^{\nu_0 d \frac{p_1 \tau_1}{p_2}q_2} \leq 	\left(\sum_{m:\ Q_{j,m}\subset Q_{\nu_0,0}} |t_{i,j,m}|^{p_2} \right)^{\frac{q_2}{p_2}} 2^{-j d\frac{p_1 \tau_1}{p_2}q_2}, 
	\end{equation*}
	which yields
	\begin{align*}
		\| \{2^{j(s_1+\frac{d}{2}-\frac{d}{p_1}+d\tau_1)}\, \gamma_j\}_{j \in \no} \mid \ell_{q_2}\| &\leq C \left\{ \sum_{j=0}^\infty 2^{j(s_2+\frac{d}{2}-\frac{d}{p_2})q_2} \, 2^{\nu_0 d\frac{p_1 \tau_1}{p_2}q_2} \sum_{i=1}^{2^d-1} \left(\sum_{m:\ Q_{j,m}\subset Q_{\nu_0,0}} |t_{i,j,m}|^{p_2} \right)^{\frac{q_2}{p_2}} \right\}^\frac{1}{q_2}\\
		& \sim 2^{\nu_0 d\frac{p_1 \tau_1}{p_2}} \left\{ \sum_{j=\max\{0,j_P\}}^\infty 2^{j(s_2+\frac{d}{2}-\frac{d}{p_2})q_2}  \sum_{i=1}^{2^d-1} \left(\sum_{m:\ Q_{j,m}\subset Q_{\nu_0,0}} |t_{i,j,m}|^{p_2} \right)^{\frac{q_2}{p_2}} \right\}^\frac{1}{q_2}\\
		&\leq  \|t \mid \wz{\sbtz}\| \\
		& \lesssim\|t \mid \wz{\sbte}\|\\
		&<\infty, 
	\end{align*}
	using \eqref{est1} in the last step and \cite[Lemma~3.3]{ysy} in the second, since $\tau_2<\frac{1}{p_2}$ here. This contradicts our assumption on the sequence $\{2^{j(s_1+\frac{d}{2}-\frac{d}{p_1}+\tau_1)}\, \gamma_j\}_{j \in \no}$, and completes the proof in this case.
\end{proof}

  Before we turn our interest to Triebel-Lizorkin-type spaces, we shall discuss some special case and compare it with the classical result as recalled in Remark~\ref{R-lim-B}. We concentrate on the limiting case \eqref{ddh_4-4} under the assumptions \eqref{ddh_4-3} again. Let us assume now $\tau_1=\tau_2=:\tau$, i.e., $0<\tau\leq \min\{\frac{1}{p_1},\frac{1}{p_2}\}$ with $q_i<\infty$ if $\tau=\frac{1}{p_i}$, $i=1,2$. In that case we find that \eqref{ddh_4-4} reads as $s_1-s_2 = d \max\{0, \frac{1}{p_1}-\frac{1}{p_2}\}$. Then Theorem~\ref{lim-lim} implies the following.

  \begin{corollary}\label{same_tau-B}
    	Let $0< p_i\leq\infty$, $s_i\in \real$, $0<q_i\leq\infty$,  $0< 
	\tau\leq \min\{\frac{1}{p_1}, \frac{1}{p_2}\}$, {with} $q_i<\infty$ {if} $\tau=\frac{1}{p_i}=\min\{\frac{1}{p_1},\frac{1}{p_2}\}$, $i=1,2$. 
	Assume that 
	\[
	\frac{s_1-s_2}{d}  = \gamma(\tau,\tau,p_1,p_2) =\max\left\{0, \frac{1}{p_1}-\frac{1}{p_2}\right\}.
	\]
Then the embedding 
	\begin{equation}
	\id_{\tau}: B^{s_1,\tau}_{p_1,q_1}(\Omega )\hookrightarrow B^{s_2,\tau}_{p_2,q_2}(\Omega )
	\end{equation}
	is continuous if, and only if, either $p_1<p_2$, or $p_1\geq p_2$ with $ q_1\le q_2$.
    \end{corollary}

  \begin{proof}
 The sufficiency follows from \eqref{520-1} in case of $p_1<p_2$, and from \eqref{520-5} for $p_1\geq p_2$. Note that the case \eqref{520-4} is not applicable in this situation. The necessity is implied by \eqref{ddh_4-5} in case of $p_1\geq p_2$, and the last statement in (ii) if $p_1<p_2$. Again, \eqref{ddh_4-6} is not possible in this context.   
  \end{proof}  

\begin{remark}  \label{R-same_tau-B}
Let us explicitly comment on the difference between the above result for $\tau>0$ and the classical one for $\tau=0$ as recalled in Remark~\ref{R-lim-B}. Only in case of embeddings of spaces with the same smoothness $s_1=s_2$ (and thus $p_1\geq p_2$) we have an influence of the fine parameters $q_i$, that is, $q_1\leq q_2$. This is parallel to the classical case $\tau=0$ and could thus be expected. However, what is far more surprising, is the outcome for $s_1>s_2$ and $p_1<p_2$: in contrast to the classical setting for $\tau=0$ we do not have any $q$-dependence here as long as $\tau>0$ (and small enough, such that we are still in the new Morrey-type situation, unlike in Propositions~\ref{lim-tau2-large} and \ref{lim-tau1-large}). Again this explains the special r\^ole of the hybrid parameter $\tau$ which influences both smoothness and integrability. 
\end{remark}

	\ignore{{\em Substep~2.3.} We still have to deal with the case $\frac{s_1-s_2}{d}= \frac{1}{p_1}- \tau_1 + \frac{1}{p_2}-\tau_2>0$.
\red{{\em Substep~2.3.} Let $\frac{s_1-s_2}{d}= \frac{1}{p_1}- \tau_1  - \frac{1}{p_2} + \tau_2>0$ and $\tau_1=\tau_2$.	Let $\lambda=(\lambda_j)_{j\in \nn}\in \ell_{q_1}$.  We put 
\[ t_{i,j,m}=
\begin{cases}
\lambda_j2^{-j(s_1+\frac{d}{2}- \frac{d}{p_1})} & \text{if}\qquad i=1 \quad\text{and}\quad m=0,\\
0 & \text{otherwise}. 
\end{cases}
\]
	Then 
	\[ 
	\|\lambda|\ell_{q_2}\| = \| t | b^{s_2,\tau_2}_{p_2,q_2}\| \le C \|t | b^{s_1,\tau_1}_{p_1,q_1}\| =  \|\lambda|\ell_{q_1}\|
	\]
since $ s_1- \frac{d}{p_1} =  s_2- \frac{d}{p_2}>0$. This implies $q_1\le q_2$. 

The last statement follows now immediately since in the  other cases we have always $q_1\le q_2$.}}

\begin{remark}
	Note that Proposition~\ref{taumonoton} can be obtained also as an immediate consequence of Theorem~\ref{lim-lim} and the Propositions~\ref{lim-tau1-large} and \ref{lim-tau2-large}. 
\end{remark}

\bigskip~
We collect now the counterpart of Theorem~\ref{lim-lim} for the Triebel-Lizorkin-type spaces. When $\tau_i<\frac{1}{p_i}$, $i=1,2$, the result follows immediately from \cite[Theorem~5.2]{hs14} and the coincidence of $\ft$ and $\MF$ spaces if $\tau=\frac{1}{p}-\frac{1}{u}$, and it reads as follows.

\begin{corollary}\label{lim-lim-c}
Let $0< p_1,p_2<\infty$, $s_i\in \real$, $0<q_i\leq\infty$,  $0\leq 
\tau_i< \frac{1}{p_i}$, 
$i=1,2$. 
Assume that 
\[
\frac{s_1-s_2}{d}  = \critical =\max\left\{0, \frac{1}{p_1}-\tau_1-\frac{1}{p_2}+\max\left\{\tau_2,\frac{p_1}{p_2}\tau_1\right\}\right\}.
\]
\begin{itemize} 
	\item[{\bfseries\upshape (i)}] The embedding 
	\begin{equation}\label{3.07.20}
	\id_{\tau}: \fte(\Omega )\hookrightarrow \ftz(\Omega )
	\end{equation}
	is continuous if one of the following conditions holds:
	\begin{align}
	& \frac{1}{p_1}-\frac{1}{p_2}> \tau_1-\tau_2
	\label{520-1f}\\
	\text{or}\qquad &\frac{1}{p_1}-\frac{1}{p_2}\le \tau_1-\tau_2\quad 
	\text{and}\quad  q_1\le \min\bigg\{1,\frac{p_1}{p_2}\bigg\}q_2\label{520-2f}.
	\end{align}
	\item[{\bfseries\upshape (ii)}] 
	If there is a continuous embedding $\id_{\tau}$ in \eqref{3.07.20}, then the parameters satisfy the condition \eqref{520-1f}, or \eqref{520-2f} holds with $q_1\leq q_2$. 	
\end{itemize}	
\end{corollary}

We return to the situation $\tau_1=\tau_2>0$ studied in Corollary~\ref{same_tau-B}, but now in case of $F$-spaces.

\begin{corollary}\label{same_tau-F}
    	Let $0< p_i<\infty$, $s_i\in \real$, $0<q_i\leq\infty$, $i=1,2$, and $0\leq 	\tau< \min\{\frac{1}{p_1}, \frac{1}{p_2}\}$. 
	Assume that 
	\[
	\frac{s_1-s_2}{d}  = \gamma(\tau,\tau,p_1,p_2) =\max\left\{0, \frac{1}{p_1}-\frac{1}{p_2}\right\}.
	\]
Then the embedding 
	\begin{equation}
	\id_{\tau}: F^{s_1,\tau}_{p_1,q_1}(\Omega )\hookrightarrow F^{s_2,\tau}_{p_2,q_2}(\Omega )
	\end{equation}
	is continuous if, and only if, either $p_1<p_2$, or $p_1\geq p_2$ with $ q_1\le q_2$.
    \end{corollary}

\begin{proof}
This is an immediate consequence of Corollary~\ref{lim-lim-c} when $\tau>0$ and of the classical situation when $\tau=0$, cf.  Remark~\ref{R-same_tau-B}.
\end{proof}

\begin{remark}
  In contrast to Remark~\ref{R-same_tau-B} concerning Besov-type spaces, we thus obtain the natural counterpart of the well-known classical situation ($\tau=0$) to the situation $\tau>0$, recall Remark~\ref{R-lim-B}.
\end{remark}

We study now some more possible situations regarding Triebel-Lizorkin-type spaces. 
\begin{corollary}\label{lim-lim-c1}
Let $0< p_1,p_2<\infty$, $s_i\in \real$, $0<q_i\leq\infty$,  $0\leq 
	\tau_i\leq \frac{1}{p_i}$, {with} $q_i<\infty$ {if} $\tau_i=\frac{1}{p_i}$
	$i=1,2$. 
	Assume that 
	\[
	\frac{s_1-s_2}{d}  = \critical =\max\left\{0, \frac{1}{p_1}-\tau_1-\frac{1}{p_2}+\max\left\{\tau_2,\frac{p_1}{p_2}\tau_1\right\}\right\}.
	\]
	\begin{itemize} 
	\item[{\bfseries\upshape (i)}] The embedding \eqref{3.07.20} is continuous if one of the following conditions holds:
	\begin{align}
	&\tau_1=\frac{1}{p_1},\quad {\tau_2  \leq \frac{1}{p_2}}\quad 
	\text{and}\quad  q_1\le q_2,\label{520-3f}\\
	\text{or}\qquad &\tau_1<\frac{1}{p_1}\quad \text{and}\quad\tau_2=\frac{1}{p_2}. 
	\label{520-4f}
	\end{align}
	\item[{\bfseries\upshape (ii)}] 
	 If the embedding \eqref{3.07.20} is continuous and $\tau_1=\frac{1}{p_1}$ and $\tau_2=\frac{1}{p_2}$, then $q_1\leq q_2$. Moreover, if $\tau_1=\frac{1}{p_1}$ and $\tau_2<\frac{1}{p_2}$, then the continuity of the embedding \eqref{3.07.20} implies $q_1 \leq \max\{p_2,q_2\}$.
\end{itemize}
\end{corollary}

\begin{proof}
	The case $\tau_1 = \frac{1}{p_1}$ and $\tau_2 = \frac{1}{p_2}$ can be reduced to Theorem~\ref{lim-lim} due to \eqref{ftbt}. So we are left with the cases $\tau_1=\frac{1}{p_1}$ and $\tau_2<\frac{1}{p_2}$ in \eqref{520-3f}, and $\tau_1<\frac{1}{p_1}$ and $\tau_2=\frac{1}{p_2}$ in \eqref{520-4f}. In both cases we can use the coincidence \eqref{ftbt}. To prove that the condition \eqref{520-4f} is sufficient we  first take a sufficiently small number $q_3$ such that $\tau_1<\frac{1}{q_3}$. We consider the following factorisation 
	\begin{align}
	F^{s_1,\tau_1}_{p_1,q_1}(\Omega)\hookrightarrow  B^{s_1,\tau_1}_{p_1,\infty}(\Omega) \hookrightarrow B^{s_2,1/q_3}_{q_3,q_3}(\Omega) = F^{s_2,1/p_2}_{p_2,q_3}(\Omega)
	\end{align}
	since 
	\[ \frac{s_1-s_2}{d} = \gamma(\tau_1,\frac{1}{p_2},  p_1,p_2)= \frac{1}{p_1}-\tau_1= \gamma(\tau_1,\frac{1}{q_3},  p_1,q_3) > 0. \]
	Please note that the continuity of the second embedding follows from Theorem~\ref{lim-lim} since the condition \eqref{520-1} is satisfied.  By elementary embeddings the statement  holds for any $q_2\ge q_3$. 
	
	Regarding the case $\tau_1 = \frac{1}{p_1}$ and $\tau_2< \frac{1}{p_2}$, we have now $\critical=0$, so $s_1=s_2$. 
	Then 
	\begin{align}
	F^{s_1,1/p_1}_{p_1,q_1}(\Omega)    =   B^{s_1,1/q_1}_{q_1,q_1}(\Omega)\hookrightarrow B^{s_2,\frac{1}{q_2}}_{q_2,q_2}(\Omega)= F^{s_2,\frac{1}{p_2}}_{p_2,q_2}(\Omega) \hookrightarrow F^{s_2,\tau_2}_{p_2,q_2}(\Omega),
	\end{align}
	where we made use of Proposition~\ref{taumonoton} in the last embedding. Moreover, 
	\[  \gamma\left(\frac{1}{q_1}, \frac{1}{q_2}, q_1,q_2\right) = 0, 
	\]
	so the statement follows once more from \eqref{ftbt} and Theorem~\ref{lim-lim}. 
	
	Now we come to the necessity part of our result. The necessity of the condition $q_1\le q_2$ in the case of $\tau_1=\frac{1}{p_1}$ and $\tau_2=\frac{1}{p_2}$ follows from the second part of Theorem~\ref{lim-lim}. So we are left with the case $\tau_1 = \frac{1}{p_1}$ and $\tau_2< \frac{1}{p_2}$. Here we can use the following factorisation
	\begin{align} 
	B^{s_1,1/q_1}_{q_1,q_1}(\Omega)= F^{s_1,1/p_1}_{p_1,q_1}(\Omega)    \hookrightarrow  F^{s_2,\tau_2}_{p_2,q_2}(\Omega) \hookrightarrow 
	B^{s_2,\tau_2}_{p_2,\max\{p_2,q_2\}}(\Omega). 
	\end{align}
	Now Theorem~\ref{lim-lim} implies $q_1\le \max\{p_2,q_2\}$ since 
	\[  \gamma\left(\frac{1}{p_1}, \tau_2, p_1,p_2\right) = \gamma\left(\frac{1}{q_1},\tau_2,  q_1,p_2\right) =0. \]
\end{proof}

 \begin{remark}
   We return to the special situation of \eqref{bmo=ft=bt-O} as a a source or target space of $\id_\tau$. In continuation of Remarks~\ref{R-bmo-1} and \ref{R-bmo-2} we now concentrate on the situation covered by Theorem~\ref{lim-lim} and Corollary~\ref{lim-lim-c1}, that is, we always assume now $\tau\leq \frac1p$ with $q<\infty$ if $\tau=\frac1p$.
First we deal  with 
    \[
    \id_\tau : \bmo(\Omega)\hookrightarrow \at(\Omega) 
    \]
    such that the limiting situation reads as $s = 0$ in that case. Then Theorem~\ref{lim-lim} and Corollary~\ref{lim-lim-c1} imply that $\id_\tau$ is continuous if
    \[
    s=0, \quad \tau\leq \frac1p, \quad \text{and}\quad \begin{cases} q\geq \max\{p,2\}, & A=B, \\ q\geq 2, & A=F. \end{cases}
    \]
Regarding the necessity, when $A=B$, the condition $q\geq 2$ is necessary for the continuity of $\id_\tau$. In addition, when $A=F$, then $q\geq 2$ is also necessary when $\tau=\frac1p$, while for $\tau<\frac1p$ the condition $\max\{p,q\}\geq 2$ is necessary. Note that one can also use \eqref{bmo=ft=bt-O} and Proposition~\ref{taumonoton} for the  sufficiency argument. In the second setting, 
    \[
\id_\tau: \at(\Omega)\hookrightarrow \bmo(\Omega),
    \]
    the limiting case means $s=d(\frac1p-\tau)$.
    Then for the continuity of $\id_\tau$ it is sufficient that
    \[ \text{either}\quad \tau<\frac1p, \qquad \text{or}\qquad\tau=\frac1p\quad\text{and}\quad\begin{cases} q\leq \min\{p,2\}, & A=B, \\ q\leq 2, & A=F, \end{cases} \]
where for $\tau=\frac1p$ the condition $q\leq 2$ is also necessary for the continuity of $\id_\tau$. 
  \end{remark}

\begin{remark}\label{R-lim-hybrid}
  Note that one can formulate counterparts of our above embedding results for spaces of type $\at(\Omega)$ in terms of the hybrid spaces $L^r\A(\Omega)$ as introduced in Remark~\ref{T-hybrid} using the coincidence \eqref{hybrid=tau}.
We define the spaces $L^r\A(\Omega)$ by restriction, parallel to the approach in Definition~\ref{tau-spaces-Omega}. 
  In \cite[Remark~3.4]{ghs20} we have explicated the compactness condition for the embedding
  \[
  \id_L : L^{r_1}\Ae(\Omega) \hookrightarrow L^{r_2}\Az(\Omega),
  \]
which in the special case $r_1=r_2=r$ reads as
\[
s_1-s_2 > \begin{cases} 0, & r\geq 0, \\ r\, \min\{\frac{p_1}{p_2}-1,0\}, & r<0,\end{cases}
\]
where $0<p_i<\infty$, $0<q_i\leq \infty$, $s_i\in\real$, $i=1,2$, and 
$-d \min\{\frac{1}{p_1}, \frac{1}{p_2}\} \leq r<\infty$ is always assumed. For convenience we only discuss this special setting $r_1=r_2=r$ and $A=B$ here, but the other cases can be done in a parallel way.

So the limiting case for the continuity of the embedding is just
\beq \label{lim-hybrid}
s_1-s_2 = \begin{cases} 0, & r\geq 0, \\ r\, \min\{\frac{p_1}{p_2}-1,0\}, & r<0.\end{cases}
\eeq
\end{remark}

\begin{corollary}\label{same_tau-hybrid}
  Let $0< p_i<\infty$, $s_i\in \real$, $0<q_i\leq\infty$, $i=1,2$, and
  $-d \min\left\{\frac{1}{p_1}, \frac{1}{p_2}\right\}< r< \infty$.  
  Let  
	\begin{equation}
	\id_L: L^r\Be(\Omega )\hookrightarrow L^r\Bz(\Omega).
	\end{equation}
        \bli
        \item[\upshape{\bfseries (i)}]
If $r>0$, then $\id_L$ is continuous if, and only if, $s_1\geq s_2$.
        \item[\upshape{\bfseries (ii)}]
          Let $r=0$.
          \bli
        \item[\upshape{\bfseries (ii$_a$)}]
          If $q_2=\infty$, then $\id_L$ is continuous if, and only  if, $s_1\geq s_2$.
        \item[\upshape{\bfseries (ii$_b$)}]
          If $q_2<\infty$ and $q_1=\infty$, then $\id_L$ is continuous if, and only if, $s_1>s_2$.           
        \item[\upshape{\bfseries (ii$_c$)}]
          If $q_i<\infty$, $i=1,2$, then $\id_L$ is continuous if $s_1>s_2$ or $s_1=s_2$ and $q_1\leq \min\{1, \frac{p_1}{p_2}\}q_2$. Conversely, if $\id_L$ is continuous, then either $s_1>s_2$ or $s_1=s_2$ and $q_1\leq q_2$.
          \eli
        \item[\upshape{\bfseries (iii)}]
          Let $r<0$.
           \bli
         \item[\upshape{\bfseries (iii$_a$)}]
 Assume that $p_1\geq  p_2$. \ignore{Then $\id_L$ is continuous if $s_1>s_2$ or $s_1=s_2$ and $q_1\leq	\frac{p_1}{p_2} q_2$. Conversely, if $\id_L$ is continuous, then either $s_1>s_2$ or $s_1=s_2$ and $q_1\leq q_2$.}  Then $\id_L$ is continuous if, and only if, $s_1>s_2$ or $s_1=s_2$ and $q_1\leq q_2$.
         \item[\upshape{\bfseries (iii$_b$)}]
           Assume $p_1< p_2$. Then $\id_L$ is continuous if $s_1-s_2>r(\frac{p_1}{p_2}-1)$, or $s_1-s_2=r(\frac{p_1}{p_2}-1)$ and $q_1\leq \frac{p_1}{p_2} q_2$. Conversely, the continuity of $\id_L$ implies $s_1-s_2>r(\frac{p_1}{p_2}-1)$, or $s_1-s_2=r(\frac{p_1}{p_2}-1)$ and $q_1\leq  q_2$.
           \eli
\eli
    \end{corollary}

\begin{proof}
  In view of Remark~\ref{R-lim-hybrid} we only need to consider the limiting case \eqref{lim-hybrid}, the rest is covered by Theorem~\ref{cont-tau} and the coincidence \eqref{hybrid=tau}, extended to spaces on domains. Then (i) and (ii$_a$) are covered by Proposition~\ref{lim-tau2-large} together with $r=d(\tau_i-\frac{1}{p_i})$, $i=1,2$. Likewise (ii$_b$) is a consequence of Proposition~\ref{lim-tau1-large} since there is no continuous embedding in the limiting case $s_1=s_2$. Part (ii$_c$) follows from Theorems~\ref{lim-lim} {and \ref{cont-tau}}, as well as part (iii).
\end{proof}

We finish our paper by collecting some immediate extensions of Theorems~\ref{B-rn} and \ref{F-rn} regarding the embeddings on $\rn$. The first result improves part (b) of Theorem~\ref{B-rn}~(iii), and it follows from Theorem~\ref{Th:ext} and Theorem~\ref{lim-lim}. 
\begin{corollary}\label{Cor1-rn}
	Let $0< p_1,p_2\leq\infty$, $s_i\in \real$, $0<q_i\leq\infty$,  $0\leq 
	\tau_i\leq \frac{1}{p_i}$, {with} $q_i<\infty$ {if} $\tau_i=\frac{1}{p_i}$, $i=1,2$. If
	\begin{equation}\label{id:B-rn}
		\id_{\tau}^{\rn}: \bte(\rn) \hookrightarrow \btz(\rn)
	\end{equation}
	is continuous and $\frac{s_1-s_2}{d}=\frac{1}{p_1}-\tau_1-\frac{1}{p_2}+ \frac{p_1}{p_2}\tau_1>0$, then
	\[ q_1\leq q_2.
	\]
\end{corollary}
\begin{proof}
	By Theorem~\ref{Th:ext}, we have
	\[
	\bte(\Omega) \xrightarrow{\ext} \bte(\rn) \xrightarrow{\id_\tau^{\rn}} \btz(\rn) \xrightarrow{\re} \btz(\Omega),
	\]
	i.e., $\id_\tau^\Omega = \re \circ \id_\tau^{\rn} \circ \ext$. Hence, the continuity of $\id_\tau^{\rn}$ implies the continuity of $\id_\tau^\Omega$, which, in turn, by Theorem~\ref{lim-lim} implies $q_1\leq q_2$.
\end{proof}

Next we show that, in case of $\tau_1=\tau_2$, the embedding $\id_\tau^{\rn}$ in \eqref{id:B-rn} holds under weaker assumptions on the parameters than the ones stated in Theorem~\ref{B-rn}-(a). Namely, in this case, we do not need any condition on the parameters $q_1, q_2$. 
\begin{corollary} 
	Let $0< p_1,p_2\leq\infty$, $s_i\in \real$, $0<q_i\leq\infty$,  $0\leq 
	\tau_i\leq \frac{1}{p_i}$, {with} $q_i<\infty$ {if} $\tau_i=\frac{1}{p_i}$, $i=1,2$. If
	\[
	\frac{s_1-s_2}{d}= \frac{1}{p_1}-\tau_1-\frac{1}{p_2}+ \tau_2>0 \quad \mbox{and} \quad \tau_1=\tau_2,
	\]
	then the embedding $\id_\tau^{\rn}$ in \eqref{id:B-rn} is continuous.
\end{corollary} 
\begin{proof}
	This result can be proved in the same way as its counterpart for embeddings on domains, cf. Substep 1.2 of the proof of Theorem~\ref{lim-lim}. Therefore, we omit {the argument}  here.
\end{proof}
Lastly, we turn to the Triebel-Lizorkin-type spaces and state a result which gives us sufficient and necessary conditions for the continuity of the embedding on $\rn$, when $\tau_2$ is large and $\tau_1$ is small. Specifically, we assume that
\begin{equation}\label{cond-tau}
\tau_2 \geq \frac{1}{p_2}\,\,\,\, \mbox{with} \,\,\,\, q_2=\infty \,\,\,\, \mbox{if} \,\,\,\, \tau_2=\frac{1}{p_2} \qquad \mbox{and} \qquad \tau_1\leq \frac{1}{p_1} \,\,\,\, \mbox{with} \,\,\,\, q_1<\infty \,\,\,\, \mbox{if} \,\,\,\, \tau_1=\frac{1}{p_1},
\end{equation}
as this case that was not considered in Theorem~\ref{F-rn}-(i). 
\begin{corollary}
	Let $0< p_1,p_2<\infty$, $s_i\in \real$, $0<q_i\leq\infty$,  $\tau_i\geq 0$, $i=1,2$. Assume that condition \eqref{cond-tau} holds. Then the embedding
	\begin{equation}\label{id:F-rn}
	\id_{\tau}^{\rn}: \fte(\rn) \hookrightarrow \ftz(\rn)
	\end{equation}
	holds if, and only if, $\quad \displaystyle\frac{s_1-s_2}{d}\geq \frac{1}{p_1}- \tau_1 - \frac{1}{p_2}+\tau_2$. 
\end{corollary}
\begin{proof}
	The sufficiency part follows from the fact that 
	\[
	\fte(\rn)\hookrightarrow B^{s_1+d(\tau_1-\frac{1}{p_1})}_{\infty, \infty}(\rn) \hookrightarrow B^{s_2+d(\tau_2-\frac{1}{p_2})}_{\infty, \infty}(\rn) = \ftz(\rn),
	\]
	due to \eqref{010319} and Proposition~\ref{yy02}, and the corresponding result for the classical Besov spaces. 
	
	For the necessity, we use a similar argument as in the proof of Corollary~\ref{Cor1-rn}, via the extension operator from Theorem~\ref{Th:ext}. In this case, Theorem~\ref{cont-tau}-(i) and Proposition~\ref{lim-tau2-large} will give us the complete result. 
\end{proof}




\vspace{1cm}
\noindent
Helena F.~Gon\c{c}alves\\
Institute of Mathematics, Friedrich Schiller University Jena, 07737 Jena, Germany\\ E-mail: helena.goncalves@uni-jena.de \\\\
\noindent
Dorothee D.~Haroske \\
Institute of Mathematics, Friedrich Schiller University Jena, 07737 Jena, Germany\\
E-mail: dorothee.haroske@uni-jena.de \\\\
\noindent
Leszek Skrzypczak \\
Faculty of Mathematics and Computer Science, Adam Mickiewicz University,\\ ul. Uniwersytetu Pozna\'nskiego 4, 61-614 Pozna\'n, Poland \\
E-mail: leszek.skrzypczak@amu.edu.pl

\end{document}